\documentclass[12pt,a4paper]{amsart}
\usepackage{amsfonts}
\usepackage{amsthm}
\usepackage{amsmath}
\usepackage{xcolor}
\usepackage{amssymb}
\usepackage{mathtools} \usepackage{amscd}
\usepackage{dsfont} \usepackage[latin2]{inputenc}
\usepackage{t1enc}
\usepackage[mathscr]{eucal}
\usepackage{indentfirst}
\usepackage{graphicx}
\usepackage{graphics}
\usepackage{hyperref}
\usepackage{epic}
\numberwithin{equation}{section}
\usepackage[margin=2.9cm, marginparwidth=2.4cm]{geometry}
\usepackage{epstopdf} 
\usepackage{enumerate}
\usepackage{tikz}
\usetikzlibrary{hobby}

\theoremstyle{plain}
\newtheorem{Th}{Theorem}[section]
\newtheorem{Lemma}[Th]{Lemma}
\newtheorem{Cor}[Th]{Corollary}
\newtheorem{Prop}[Th]{Proposition}

 \theoremstyle{definition}

\newtheorem{Rem}[Th]{Remark}
\newtheorem{?}[Th]{Problem}

\newcommand{\Pp}{\mathbb{P}}
\newcommand{\Prob}[1]{\mathbb{P}\{ #1 \}}

\newcommand{\Probx}[2]{\mathbb{P}^{#1}\{ #2 \}}
\newcommand{\ProbxBig}[2]{\mathbb{P}^{#1}\Big\{ #2 \Big\}}
\newcommand{\ProbBig}[1]{\mathbb{P}\Big\{ #1 \Big\}}
\newcommand{\E}[1]{\mathbb{E} #1 }
\newcommand{\EX}[2]{\mathbb{E}^{#1} #2 }
\newcommand{\EXBig}[2]{\mathbb{E}^{#1}\Big( #2 \Big)}

\newcommand{\EBig}[1]{\mathbb{E} #1 }

\newcommand{\Ind}[1]{\mathds{1}_{\{ #1\}}}

\newcommand{\re}{\operatorname{Re}}

\newcommand{\eps}{\epsilon}
\newcommand{\lam}{\lambda}

\newcommand{\tj}{{\tau_j}}

\newcommand{\lone}{{\lambda_1}}
\newcommand{\li}{{\lambda_i}}
\newcommand{\lj}{{\lambda_j}}
\newcommand{\lk}{{\lambda_k}}

\newcommand{\gi}{{\gamma_i}}
\newcommand{\gj}{{\gamma_j}}
\newcommand{\R}{\mathbb{R}}
\newcommand{\N}{\mathbb{N}}

\newcommand{\smallo}[1]{{o\big(#1\big)}}

\newcommand{\D}{\mathcal{D}} \newcommand{\Hil}{\mathcal{H}} \newcommand{\pbracket}[1]{\big[#1\big]_p}
\newcommand{\pbracketBig}[1]{\Big[#1\Big]_p}

\newcommand{\pbracketX}[2]{\big[#1\big]_{#2}}

\newcommand{\Ceps}{\mathcal{C}_\eps}
\newcommand{\Czero}{\mathcal{C}_0}

\usepackage{ulem}
\usepackage{cancel}

\newcommand{\EE}{\mathbb{E}}
\newcommand{\CC}{\mathbf{C}}
\newcommand{\pp}{\mathtt{p}}
\newcommand{\NN}{\mathcal{N}}
\newcommand{\bA}{\overline{A}}
\newcommand{\bB}{\overline{B}}
\newcommand{\YY}{\mathbf{Y}}
\newcommand{\ZZ}{\mathbf{Z}}
\newcommand{\tZZ}{\widetilde{\mathbf{Z}}}
\newcommand{\qd}[1]{\langle #1 \rangle}
\newcommand{\tZ}{\widetilde{Z}}

\newcommand{\DD}{\mathbf{D}}

\newcommand{\ab}{\mathbf{a}}

\newcommand{\rr}{h}

\newcommand{\Sph}{\mathbb{S}}

\begin{document}

\title{Long Exit Times near a Repelling Equilibrium}

\author{Yuri Bakhtin}
\author{Hong-Bin Chen}
\address{Courant Institute of Mathematical Sciences\\ New York University \\ 251~Mercer~St, New York, NY 10012 }
\email{bakhtin@cims.nyu.edu, hbchen@cims.nyu.edu }



\begin{abstract} 
For a smooth vector field in a neighborhood of a critical point with all positive eigenvalues of the linearization, we consider the associated dynamics perturbed by white noise. Using Malliavin calculus tools, we obtain polynomial asymptotics for probabilities of atypically long exit times
in the vanishing noise limit.
\end{abstract}

\maketitle
\tableofcontents

\section{Introduction}

In this paper, we continue the study of exit time distributions for diffusions obtained by
small noisy perturbations of deterministic dynamical systems
near unstable critical points. We are motivated by applications to the long-term dynamics in noisy heteroclinic networks and extensions of the work in~\cite{Bakhtin2011}, \cite{Bakhtin2010:MR2731621}, \cite{Almada-Bakhtn:MR2802310}. 

The most celebrated series of results on random perturbations of dynamical systems known as the Freidlin--Wentzell theory of metastability, see~\cite{FW}, is based on large deviation estimates and computes the asymptotics of probabilities associated with rare transitions between neighborhoods of stable equilibria. In these systems, the probability of a transition in a given finite time decays exponentially in  $\eps^{-2}$, where $\eps>0$ is the noise magnitude, so it takes time of the order of 
$\exp(c\eps^{-2})$, to realize these transitions.

In the noisy heteroclinic network setting, it turns out that rare events of interest describing atypical transitions and determining the long-term behavior of the diffusion
are 
tightly related to abnormally long stays in neighborhoods of unstable critical points. As a result, the probabilities of those events are related to the tails of the associated exit times, see a discussion of heteroclinic networks in~\cite{long_exit_time-1d-part-2}.

The probabilities we are interested in were shown to decay as a power of $\eps$ if the starting point belongs to the stable manifold of the hyperbolic critical point (saddle) in~\cite{mikami1995}. In the present paper, we provide much more precise asymptotics than the large deviation results of~\cite{mikami1995} and prove a conjecture stated in that paper.

To be more precise, for $\eps>0$, let us consider a diffusion process $X^\eps$ solving an SDE
in $\R^d$, $d\in\N$:
\begin{align}\label{eq:SDE_X}
    dX^\eps_t = b(X^\eps_t)dt + \eps \sigma(X^\eps_t)dW_t
\end{align}
with noise given by the standard multi-dimensional Wiener process $W$ and a smooth full-rank diffusion matrix $\sigma$,
started at a distance of the order of $\eps$ from the origin~$0$ which is
assumed to be an unstable critical point of the smooth vector field $b$.  Let $\lambda_1>0$ be the leading simple eigenvalue of $D b(0)$, i.e., the real parts of all other eigenvalues are less than $\lambda_1$.

We are interested in the exit time $\tau_\eps$ from a domain $\DD$ containing $0$ and having a smooth boundary. The first results showing that the exit times typically behave like
$T_\eps=\frac{1}{\lambda_1}\log\frac{1}{\eps}$ plus $O(1)$ corrections,
 were obtained in~\cite{Kifer1981} and ~\cite{Day95}.
Namely,
it was shown  in~\cite{Kifer1981} that  $\frac{\tau^\eps}{T_\eps} \stackrel{\Pp}{\to} 1$, $\eps\to 0$, and
in~\cite{Day95}, the limiting distribution of
$\tau^\eps-T_\eps$ 
as $\eps\to0$  was 
found. The distributions of exit locations were
studied in~\cite{Eizenberg:MR749377},~\cite{Bakhtin2011}, and (for the case where $Db(0)$ is a Jordan block) in \cite{B-PG:doi:10.1142/S0219493719500229}.

In~\cite{mikami1995}, probabilities of atypical deviations of $\tau_\eps$ from $T_\eps$ were studied. It was proved that in the $1$-dimensional situation $(d=1)$, for any $\rr>1$,
\begin{equation}
\label{eq:mikami1}
\lim_{\eps\to 0} \frac{\log \Prob{\tau_\eps>\rr T_\eps }}{\log \eps} =\rr-1.
\end{equation}
and a combination of results in~\cite{Kifer1981}  and~\cite{mikami1995} gives that for 
all $d\ge 1$ and 
every $\rr>1$ there are finite positive numbers $\mu_-(\rr),\mu_+(\rr)>0$ such that
\begin{equation}
\label{eq:mikami2}
\mu_-(\rr)\le \liminf_{\eps\to 0} \frac{\log \Prob{\tau_\eps>\rr T_\eps }}{\log \eps} \le \limsup_{\eps\to 0} \frac{\log \Prob{\tau_\eps>\rr T_\eps }}{\log \eps} < \mu_+(\rr).
\end{equation}
In~\cite{mikami1995} it is actually conjectured that 
\begin{equation}
\label{eq:mikami-conjecture}
\mu_-(\rr)=\mu_+(\rr)=\mu(\rr),
\end{equation}
where
\begin{equation}
\label{eq:mikami3}
\mu(\rr)= \sum_{j=1}^d\left(\left(\frac{\rr\re \lj }{\lambda_1} -1\right)\vee 0\right),
\end{equation}
and $\lambda_1,\ldots,\lambda_d$ in this formula are the eigenvalues of $D b(0)$.

In~\cite{long_exit_time} and~\cite{long_exit_time-1d-part-2}, the  logarithmic asymptotics
of~\eqref{eq:mikami1} for the $1$-dimensional situation  was improved and it was shown that for 
any $\rr>1$, for a range of deterministic initial conditions~$X^\eps_0=x$ near $0$, 
\begin{equation}
\label{eq:recalling-old-result}
\Prob{\tau_\eps>\rr T_\eps}=\psi(x)\eps^{\rr-1}(1+o(1)),\quad \eps\to 0,
\end{equation}
and the coefficient $\psi(x)>0$ was computed explicitly. The paper~\cite{long_exit_time} was based on Malliavin calculus techniques and~\cite{long_exit_time-1d-part-2} used more elementary tools.

In the present paper, we consider the situation where $d\in\N$ is arbitrary and the eigenvalues of $\nabla b(0)$ are real and satisfy $\lambda_1>\lambda_2>\ldots>\lambda_d>0$. For this case, we prove the conjecture of~\cite{mikami1995} showing that relations  \eqref{eq:mikami2}--\eqref{eq:mikami3} hold true. In fact, instead of the logarithmic equivalence in \eqref{eq:mikami2}, we prove stronger estimates similar to~\eqref{eq:recalling-old-result} extending the latter to the higher-dimensional setting.
For domains~$\DD$ of a special type (preimages of rectangular domains under a linearizing
conjugacy), our Theorem~\ref{thm:main} states that there is $p>0$ such that, uniformly over deterministic initial conditions  $X^\eps_0=x$ at distance of the order of $\eps$ from $0$,
\begin{equation*}
\Prob{\tau_\eps>\rr T_\eps} = \psi_{\rr}(x)\eps^{\mu(\rr)}(1+ o(\eps^p)),
\end{equation*}
with an explicit expression for the coefficient $\psi_{\rr}(x)>0$. In fact, we prove a more general estimate on the tail of  
$\tau_\eps$.

The idea of the proof is the following. We treat the dynamics described by~\eqref{eq:SDE_X} as a perturbation of the linear dynamics given by the linearization of $b$ at $0$. For truly linear dynamics with additive noise the solution is given by stochastic It\^o integrals of deterministic quantities. Thus it is a Gaussian process allowing for a direct computation which, in fact, was behind the conjecture \eqref{eq:mikami2}--\eqref{eq:mikami3} of~\cite{mikami1995}. The main difficulty is to lift this computation to the general nonlinear situation. In particular, similarly 
to~\cite{long_exit_time} we choose to work with Malliavin calculus tools in order to estimate
densities of random variables that we want to treat as perturbations of Gaussian ones.
Unlike~\cite{long_exit_time}, we use results of~\cite{bally2014} to estimate the discrepancy between the Gaussian densities and the perturbed ones. These estimates are valid only for evolution times of the order of $\theta\log\eps^{-1}$ with small values of $\theta$, so we have to apply them sequentially multiple times in order to get to $\rr T_\eps$, thus creating an iteration scheme similar to that of~\cite{long_exit_time}.

The analysis for more general domains  can be partially reduced to the special domains 
defined above via the rectifying conjugacy. We can obtain, see Corollary~\ref{cor:1}, that there are constants $\phi_\pm(x)$  such that
\begin{equation*}
    \phi_-(x)\eps^{\mu(\rr)}(1+o(\eps^p))\leq \Prob{\tau_\eps>\rr T_\eps+r(\eps)} \leq \phi_+(x)\eps^{\mu(\rr)}(1+o(\eps^p)).
\end{equation*}
The slight discrepancy between the upper and lower estimates is due to the fact that the travel time along the drift vector field between the boundaries of domains immersed into one another depends on the starting point on the boundary. We give a slightly more precise result (Corollary~\ref{prop:2}) that takes these travel times into account and note here that further progress in understanding of exit times for general domains will be achieved as more information on the geometric  properties of the exit location distribution becomes available. The  asymptotics of the exit location distribution will be addressed in our forthcoming work.

The paper is organized as follows. In Section~\ref{sec:setting}, we give a technical description of the setting and state our main results precisely. The proof is spread over Sections~\ref{sec:proof-main} through~\ref{section:density_estimate}.
The main result is derived from the comparison to the linearized problem in Section~\ref{sec:proof-main}. An iterative scheme of sequential approximations that this comparison is based on is given in Section~\ref{section:approximations}. Each step of this scheme is in turn based on a density discrepancy estimate that we derive using Malliavin calculus tools in Section~\ref{section:density_estimate}.

{\bf Acknowledgment.} YB is grateful to NSF for partial support via grant DMS-1811444.

\section{Setting and main results}
\label{sec:setting}

Let $d\in \N$ and let simply connected domains $D_1,D_2,\DD\subset \R^d$ satisfy  
\begin{equation}\label{eq:D_inclusions}
                0\in  D_1 \subset \overline{D_1} \subset \DD \subset \overline{\DD}\subset D_2.
\end{equation}
We consider a $C^5$ vector field $b:\R^d \to \R^d$  and the flow $(S^t)$ generated by $b$:
\begin{align}\label{eq:ODE}
    \begin{split}
        \tfrac{d}{dt}S^tx & = b(S^tx),\\
    S^0x &= x.
    \end{split}
\end{align}
Since we are interested in the dynamics inside $\DD$, by adjustments outside $\DD$, we assume $b$ and its derivatives are bounded.
We assume that the following conditions hold:
\begin{itemize}
    \item  $b(x)=\mathbf{a}x+q(x)$ where 
        \begin{itemize}
            \item $|q(x)|\leq C_q|x|^2$ with a positive constant $C_q$,
            \item $\mathbf{a}$ is a $d\times d$ diagonal matrix with real entries $\lambda_1>\lambda_2>...>\lambda_d>0$;
        \end{itemize}
    \item for all open sets $D_0$ satisfying $0\in D_0\subset D_1$,
            \begin{equation}\label{eq:condition_domain_D}               
\sup_{x\in \partial D_0} t_{D_2}(x) < \infty,
            \end{equation}
where 
            \begin{equation}t_{D}(x) = \inf\{t>0: S^tx \not\in D\}, \quad D\subset \R^d,\ x\in\R^d.
            \label{eq:t_D}
            \end{equation}
\end{itemize}

\medskip

For brevity we will denote the vector filed  given by $x\to\mathbf{a}x$ by $\ab$.
By the Hartman--Grobman Theorem (c.f. Theorem 6.3.1 from \cite{Dynamics}), there 
is an open neighborhood $O$ of $0$ and a homeomorphism $f:O\to f(O)$ conjugating the flow $S$ generated by the vector field $b$ to the linear flow  generated by $ \mathbf{a}$, namely, 
\begin{align}\label{eq:conjugation}
    \frac{d}{dt}f(S_tx)=\mathbf{a}f(S_tx).
\end{align}
\begin{itemize}
    \item in addition, we assume that $f$ is a $C^5$ diffeomorphism.

\end{itemize}    
\begin{Rem}
Due to \cite{Sternberg-1957:MR96853},  for this  $C^5$ conjugacy condition to hold in our setting, it suffices to require (i) a smoothness condition:
$b$ is $C^k$ for some $k\ge 5 \vee( \lambda_1 /  \lambda_n)$, and (ii) a no-resonanse condition:
\[\lambda_k\ne m_1\lambda_1+\ldots+m_d \lambda_d\]
for all $k=1,\ldots,d$ and all  nonnegative integer coefficients $m_1,\ldots, m_d$ satisfying $m_1+\ldots+m_d\ge 2$.
\end{Rem}

The vector field $\mathbf{a}$ is the pushforward of $b$ under $f$, and since $\mathbf{a}$ is diagonal, $f$ can be chosen to satisfy
\begin{align}\label{eq:diffeo-norm}
    f(0)=0,\quad Df(0)=I,
\end{align}
where $I$ is the identity matrix.

\bigskip

We are  interested in the limiting behavior of random perturbations of the ODE~\eqref{eq:ODE} given by  the SDE~\eqref{eq:SDE_X} as $\eps$ tends to $0$. In~\eqref{eq:SDE_X},
\begin{itemize}
    \item $\eps\in(0,1)$ is the  noise amplitude parameter;
    \item $(W_t,\mathcal{F}_t)$ is a standard $n$-dimensional Wiener process with $n\geq d$;
    \item $\sigma$ is a map from $\R^d$ into the space of $d\times n$ matrices satisfying
        \begin{itemize}
            \item $\sigma$ is $C^3$ (and, by adjustments outside $\DD$, we may assume that $\sigma$ and its derivatives are bounded),
            \item $\sigma(0):\R^n \to \R^d$ is surjective.
        \end{itemize}
\end{itemize}
To simplify the notation, we often suppress the dependence on $\eps$. In particular, we often write $X_t$ instead of $X^\eps_t$.

\bigskip

We need some definitions to state our main result. We start by describing the exit event:
\begin{itemize}
    \item for a measurable set  $A \subset \R^d$, we define the exit time 
    \begin{align}\label{eq:def_tau_A}
        \tau_A=\inf\{t>0: X_t\not \in A \};
    \end{align}
    \item for $L^j_-, L^j_+\in \R$, $j=1,2,\dots, d$, we define $\mathfrak{R}=f^{-1}\big(\prod_{j=1}^d [L_-^j, L_+^j]\big)\subset O$ with $ 0 \in \mathring{\mathfrak{R}}$ to be such a set that its preimage under $f$ is a box and that its interior contains the origin (see Figure \ref{figure:1});
    \item for any $q>0$, let $r:[0,1] \to \R$ be any function satisfy
    \begin{align}\label{eq:condition_for_r(eps)}
    |r(\eps)-r(0)|=\mathcal{O}(\eps^q).
    \end{align}
\end{itemize}
The theorem is concerned with events of the form
\begin{align*}
    \{\tau_\mathfrak{R}>\alpha \log\eps^{-1}+r(\eps)\}
\end{align*}
for some $\alpha\geq0$. As $\tau_\mathfrak{R}$ is of order $\alpha \log\eps^{-1}$, the term $r(\eps)$ is interpreted as a small perturbation. 

\smallskip

\begin{figure}
    \centering
    \begin{tikzpicture}[use Hobby shortcut,closed=true]
    \filldraw [black] (0,0) circle (1pt);
    \filldraw [black] (6,0) circle (1pt);
    \node [right] at (0,0.2) {$0$};
    \node [right] at (0,-0.5) {$\mathfrak{R}$};
    \node [right] at (-0.2,1.04) {$O$};
    \node [right] at (0.1, 2) {$\mathbf{D}$};
    \node [right] at (6,0.2) {$0$};
    \node [right] at (5.4,-0.5) {$f(\mathfrak{R})$};
    \node [right] at (5.35,1.1) {$f(O)$};
    \node [right] at (2.98,0.25) {$f$};

    \draw [->,>=stealth] (2.3, 0) -- (4.2, 0);
    
    \draw (0,0) circle (1.42cm);
    \draw (-1.6,0)..(-1.6,0.65)..(-1.45,1.55)..(-1.3,1.8)..(0,2.6)..(1,2.4)..(1.35,2)..(1.78,1)..(1.85,0)..(1.7,-1)..(0.9,-1.9)..(0,-1.7)..(-1.5,-1.8)..(-1.6,0);

    \draw plot [smooth, tension=2] coordinates { (-1,-0.6) (-0.7,0) (-0.6,0.75)};
    
    \draw plot [smooth, tension=2] coordinates { (-0.6,0.75) (0.2,0.65) (1,0.7)};
    
    \draw plot [smooth, tension=2] coordinates { (0.65,-1) (1,-0.25) (1,0.7)};
    
    \draw plot [smooth, tension=2] coordinates { (0.65,-1) (-0.2,-1) (-1,-0.6)};
    
    \draw [yshift=-0.2cm] (5.15,0.8) -- (6.8,0.8) -- (6.8,-0.65) -- (5.15,-0.65)  -- cycle;

    \draw [yshift=-0.2cm]  (4.67,0)..(4.68,0.8)..(6.65,1.65)..(7.2,0.4)..(6.35,-1.5)..(5.4,-1.2)..(4.67,0);

\end{tikzpicture}
    \caption{The diffeomorphism $f:O\to f(O)$ maps $\mathfrak{R}$ onto $f(\mathfrak{R})$ which is a box containing $0$.}\label{figure:1}
\end{figure}
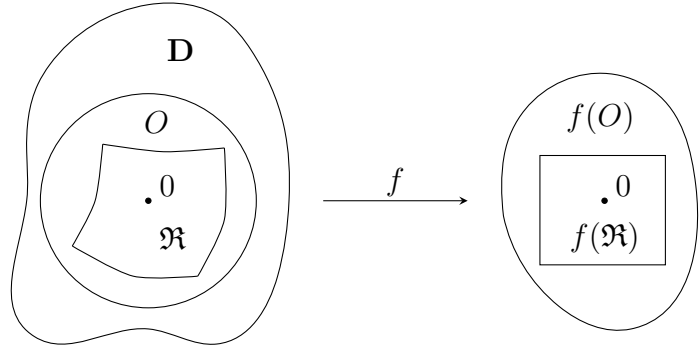

Next, we introduce definitions characterizing the decay rate:

\begin{itemize}

    \item for $\alpha >0$, let $i(\alpha)\in \{1,2,...,d+1 \}$ satisfy 
    \begin{align}\label{eq:def_i=i(alpha)}
        \tfrac{1}{\lambda_{i(\alpha)-1}}< \alpha \leq \tfrac{1}{\lambda_{i(\alpha)}}
    \end{align}
    where we agree that $\lambda_0 = \infty$ and $\lambda_{d+1}=0$;
    \item the exponent determining the power decay, as a function of $\alpha$, is given by
    \begin{align}\label{eq:beta_alpha}
    \beta(\alpha) = \sum_{j=1}^d\big((\lj\alpha -1)\vee 0\big)= \sum_{j=1}^{i(\alpha)-1}(\lj\alpha -1)= \mu(\lambda_1\alpha),
    \end{align} 
     where $\mu(\cdot)$ was defined in~\eqref{eq:mikami3}.
\end{itemize}

\smallskip

We will consider initial conditions $X_0=\eps x$ satisfying $|x|\leq K(\eps)$ for an admissible function $K(\cdot)$:
\begin{itemize}
    \item for a fixed $\alpha\geq 0$, a function $K:(0,1]\to [0,+\infty)$ is said to be admissible if it satisfies, with $i=i(\alpha)$,
    \begin{align*}
        \begin{cases}
            \lim_{\eps \to 0}\eps^{1-\lambda_{i}\alpha} K(\eps) = 0, &\quad \text{if } i\leq d\text{ and }\alpha<\tfrac{1}{\lambda_i},\\
            \lim_{\eps \to 0}\eps^{1-\lambda_{i+1}\alpha} K(\eps) = 0, &\quad \text{if }i\leq d,\text{ and }\alpha=\tfrac{1}{\lambda_i},\\
            \lim_{\eps \to 0}\eps^{1-c} K(\eps) = 0 \text{ for some $c\in (0,1)$}, &\quad \text{if } i = d+1.\\
        \end{cases}
    \end{align*}
\end{itemize}

\smallskip

Lastly, we describe the limiting object:

\begin{itemize}

    \item let $d\times d$ matrix $\Czero$ be given by
    \begin{align}\label{eq:def_C_0_Czero}
        \Czero^{jk}= \sum_{l=1}^n\frac{\sigma^j_l(0)\sigma^k_l(0)}{\lj + \lk};
    \end{align}
    \item for $x \in \R^d$ and $i=1,\ldots,d$, we define 
    \begin{align}\label{eq:projection_notation1}
        \begin{split}
            &x^{<i}=(x^1,x^2,...,x^{i-1})\in \R^{i-1},\quad x^{>i}=(x^{i+1},...,x^d)\in \R^{d-i},\\ &x^{\geq i} =(x^{i},x^{i+1},...,x^d)\in \R^{d-i+1};
              \end{split}
    \end{align} 
    \item for $\alpha \geq 0$, some small perturbation limit $r(0)\in\R$, a set $\mathfrak{R}$ and $x\in\R^d$, we define, with $i=i(\alpha)$,
  \begin{align}\label{eq:psi}
    \begin{split}
        & \psi_{\alpha, r(0),\mathfrak{R}}(x)=\\
    &\begin{cases}\displaystyle
    \frac{\prod_{j< i}(L^j_+-L^j_-)e^{-\lj r(0)}}{\sqrt{(2\pi)^d\det \Czero}}\int_{\R^{d-i+1}}e^{-\frac{1}{2}z^\intercal\Czero^{-1}z}\big|_{z^{<i}=-x^{<i}} dz^{\geq i},& \alpha <\frac{1}{\li}, \\
    \displaystyle
    \frac{\prod_{j< i}(L^j_+-L^j_-)e^{-\lj r(0)}}{\sqrt{(2\pi)^d\det \Czero}}\int_{(e^{-\li r(0)}[L^i_-,L^i_+]-x^i))\times\R^{d-i}}e^{-\frac{1}{2}z^\intercal\Czero^{-1}z}\big|_{z^{<i}=-x^{<i}}dz^{\geq i},&\alpha=\frac{1}{\li}.
    \end{cases}
    \end{split}
\end{align}
 \end{itemize}
If $i=d+1$, then  the integrals in \eqref{eq:psi} are understood to be simply  $e^{-\frac{1}{2}x^\intercal \Czero^{-1}x}$. 

We are now ready to state the main result.

\begin{Th}
\label{thm:main}
Suppose $X_t$ solves \eqref{eq:SDE_X} with $X_0=\eps x$, and $r$ satisfies \eqref{eq:condition_for_r(eps)} for some $q>0$.

There is a constant $L_0\geq 0$ such that, for every $\alpha\geq 0$, every admissible $K$, every $\mathfrak{R}\subset f^{-1}([-L_0,L_0]^d)$, the following holds
\begin{align}\label{eq:main}
    \sup_{|x|\leq K(\eps)}\big|\eps^{-\beta(\alpha)}\Prob{\tau_\mathfrak{R}>\alpha \log\eps^{-1}+r(\eps)} - \psi(x) \big| = \smallo{\eps^p}
\end{align}
for $\psi =\psi_{\alpha,r(0),\mathfrak{R}}$ and some $p=p(\alpha,q,\lambda, \sigma, f) \in (0,1)$. 
\end{Th}

\bigskip

For a general domain $\DD$, we choose $L_{\pm}^j$ small enough to guarantee $\mathfrak{R}\subset D_1$, where~$D_1$ was introduced in \eqref{eq:D_inclusions}. 
Due to~\eqref{eq:condition_domain_D}, $T_- = \inf_{z\in \partial\mathfrak{R}} t_{D_1}(z)$ and $T_+ = \sup_{z\in \partial\mathfrak{R}} t_{D_2}(z)$ are well-defined. Setting $\phi_\pm(x)= \psi_{\alpha,r(0)- T_\pm,\ \mathfrak{R}}(x)$, we obtain:

\begin{Cor}\label{cor:1}
Under the conditions of Theorem~\ref{thm:main},
\begin{align}\label{eq:cor}
    \phi_-(x) + \smallo{\eps^p} \leq \eps^{-\beta(\alpha)}\Prob{\tau_\DD>\alpha \log\eps^{-1}+r(\eps)} \leq \phi_+(x) + \smallo{\eps^p}
\end{align}
uniformly over $|x|\leq K(\eps)$.
\end{Cor}

Taking the logarithm on both both sides of \eqref{eq:cor}, we obtain:
\begin{Cor}\label{cor:2}
Under the conditions of Theorem~\ref{thm:main}, there is a constant $C>0$ such that
\begin{align*}
    \sup_{|x|\leq K(\eps)}\bigg|\frac{\log \Prob{\tau_\DD>\alpha \log\eps^{-1}+r(\eps)}}{\log \eps}-\beta(\alpha)\bigg|\leq \frac{C}{
 |\log \eps|}.
\end{align*}
\end{Cor}

\bigskip

\begin{Rem}
\begin{enumerate}
    \item When $d=1$, Proposition \ref{thm:main} is a slight improvement of the result in \cite{long_exit_time}.
    \item If $q =0$, then the above results still hold for  $p=0$. 
    \item If $X_0 = \eps \xi^\eps$ where the random variable $\xi^\eps$ satisfies $\Prob{|\xi^\eps|>K(\eps)} = \smallo{\eps^{\beta(\alpha)}}$, then \eqref{eq:main} and \eqref{eq:cor} imply, respectively, 
    \begin{gather*}
       \lim_{\eps \to 0}\big|\eps^{-\beta(\alpha)}\Prob{\tau_\mathfrak{R}>\alpha \log\eps^{-1}+r(\eps)}-\E{\psi(\xi^\eps)}\big| = 0;\\
\E{\phi_-(\xi^\eps)} + \smallo{1} \leq \eps^{-\beta(\alpha)}\Prob{\tau_\DD>\alpha \log\eps^{-1}+r(\eps)} \leq \E{\phi_+(\xi^\eps)} + \smallo{1}. 
    \end{gather*}
    \item In comparison with~\cite{mikami1995}, we make stronger smoothness assumptions on the coefficients and an additional assumption on the smoothness of the linearizing conjugacy. These assumptions are required for our Malliavin calculus approach. Namely, we must ensure that certain higher-order Malliavin derivatives of the diffusion process exist and admit useful bounds. 
In addition, we require the eigenvalues of linearization to be simple and positive.   
    In this slightly more restrictive setting, our Corollary~\ref{cor:2} improves and generalizes \cite[Theorem~1.3 and~Proposition~1.4]{mikami1995} and implies  \cite[Conjecture~1.5]{mikami1995}.

\end{enumerate}
\end{Rem}

Under additional geometric assumptions on $\DD$, more precise results than Corollary~\ref{cor:1} can be obtained.
We assume that $\DD$ has $C^1$ boundary and that $b$ intersects $\partial \DD$ transversally in the sense that $\langle n(x),b(x)\rangle >0$   for every $x\in\partial \DD$, where $n(x)$ is the outer normal of $\partial \DD$. Let us choose $L^j_\pm$ small enough to ensure $\overline{\mathfrak{R}}\subset \DD$ and recall~\eqref{eq:t_D}.

\begin{Cor}\label{prop:2}
Under the same conditions as  Theorem~\ref{thm:main} and the additional smoothness and transversality assumptions introduced in the above paragraph, we have
\begin{align*}
    \sup_{|x|\leq K(\eps)}\big|\eps^{-\beta(\alpha)}\Prob{\tau_\DD-t_\DD(X_{\tau_\mathfrak{R}})>\alpha \log\eps^{-1}+r(\eps)} - \psi (x) \big| = \smallo{\eps^p},
\end{align*}
where $\psi = \psi_{\alpha,r(0),\mathfrak{R}}$ is given in \eqref{eq:psi}.
\end{Cor}

\section{Proof of main results} \label{sec:proof-main}

Corollaries  \ref{cor:1} and~\ref{prop:2} are direct consequences of Theorem~\ref{thm:main}, our geometric assumptions, and the following standard FW large deviation estimate
which implies  that, upon exiting $\mathfrak{R}$, the process $X$ closely follows a deterministic trajectory:
\begin{Lemma}\label{lemma:FW}
    For each fixed time $T>0$, and each $\upsilon \in[0,1)$, there are $C,c$>0 such that the following holds uniformly over all initial points $X_0=\mathbf{x}$:
    \begin{align*}
        \Prob{\sup_{0\leq t\leq T}|X_t-S^t\mathbf{x}|>\eps^{\upsilon}}\leq C\exp( -c\eps^{2(\upsilon-1)}).
    \end{align*} 
\end{Lemma}

This lemma can be proved using Lipschitzness of the vector field $b$, boundedness of~$\sigma$, Gronwall's inequality, and the exponential martingale inequality (see \cite[Problem 12.10]{Bass}). The key idea can be seen at the very beginning of~\cite[Chapter~3]{FW}. 
\bigskip

The rest of this section is our  proof  of Theorem~\ref{thm:main}.

From now on we will often use  Einstein's convention of summation over matching upper and lower indices.
Let us introduce a new process $Y_t = f(X_t)$, which by It\^o's formula and \eqref{eq:conjugation} satisfies
\begin{equation}
 \label{eq:Y_SDE_before_Duhamel}
dY^i_t = \lambda^i Y^i_t dt + \eps F^i_j(Y_t) dW^j_t + \eps^2G^i(Y_t)dt,
\end{equation}
where
\begin{align*}
F_j^i(y)&=\partial_k f^i(f^{-1}(y))\sigma^k_j(f^{-1}(y)),\quad y\in f(O),\\
G^i(y)&=\frac{1}{2}\partial_{jk}^2f^i(f^{-1}(y))\langle \sigma^j(f^{-1}(y)),\sigma^k(f^{-1}(y)) \rangle,\quad y\in f(O),
\end{align*}
$\langle \cdot, \cdot \rangle$ denotes the inner product, and we set $\lambda^i = \lambda_i$ to avoid the summation
over~$i$. Note that $F,G\in C^3(f(O))$  and, due to~\eqref{eq:diffeo-norm}, we have \begin{align}\label{eq:F(0)=simga(0)}
        F(0)= \sigma(0).
    \end{align}

We shift our focus from the process $X_t$ with $X_0=\eps x$ to $Y_t=f(X_t)$ with $Y_0=\eps y = f(\eps x)$ by the following considerations. Due to~\eqref{eq:diffeo-norm}, there is a constant $C_f$ such that $|z|\leq C_f|f(z)|$ for all $z\in O$. Set $K'(\eps)=C_f^{-1}K(\eps)$. 
Therefore, for $\eps $ small with $X_0=\eps x \in O$, we have that if $|y|\leq K'(\eps)$, then $|x|\leq K(\eps)$.
Note that due to $Y_t=f(X_t)$  the exit time~$\tau_\mathfrak{R'}$ defined in \eqref{eq:def_tau_A} in terms of the process $X$ can be rewritten as
\begin{equation}
\label{eq:tau}
\tau =\inf\{t>0: Y_t \not\in \mathfrak{R}'\},
\end{equation}
where
$\mathfrak{R}'=\prod_{j=1}^d [L_-^j, L_+^j]=f(\mathfrak{R})$ (see Figure \ref{figure:1}). Hence, Theorem~\ref{thm:main} follows from the following result.
\begin{Prop}\label{Prop:Y} Suppose $Y_t$ solves \eqref{eq:Y_SDE_before_Duhamel} with $Y_0=\eps y$ and let $r$ satisfy\eqref{eq:condition_for_r(eps)}.
Then there is a constant $L_0\geq 0$ such that for each $\alpha\geq 0$ and each $K'(\eps)$ satisfying, with $i=i(\alpha)$,
\begin{align}\label{eq:condition_for_K'}
    \begin{cases}
            \lim_{\eps \to 0}\eps^{1-\lambda_{i}\alpha} K'(\eps) = 0, &\quad\text{if } i\leq d\text{ and }\alpha<\tfrac{1}{\lambda_i},\\
            \lim_{\eps \to 0}\eps^{1-\lambda_{i+1}\alpha} K'(\eps) = 0 , &\quad\text{if } i\leq d\text{ and }\alpha=\tfrac{1}{\lambda_i},\\
            \lim_{\eps \to 0}\eps^{1-c} K'(\eps) = 0 \text{ for some $c\in (0,1)$}, &\quad\text{if } i = d+1,\\
        \end{cases}
\end{align}
we have, for any set of the form  $\mathfrak{R}'=\prod_{j=1}^d [L_-^j, L_+^j]\subset O$ with $ 0 \in \mathring{\mathfrak{R}}'$ and $|L_\pm^j|\leq L_0$ for all $j=1,\ldots,d$, 
\begin{align}\label{eq:main_Y}
    \sup_{|y|\leq K'(\eps)}\big|\eps^{-\beta(\alpha)}\Prob{\tau>\alpha \log\eps^{-1}+r(\eps)} - \psi\big(\tfrac{f^{-1}(\eps y)}{\eps} \big) \big| = \smallo{\eps^p},
\end{align}
for some $p=p(\alpha,q,\lambda, \sigma, f) \in (0,1)$.

\end{Prop}

\bigskip

Let us describe the plan to prove Proposition \ref{Prop:Y}. The proof can be split into two main steps. 

The first step is to show that $\Prob{\tau>\alpha\log\eps^{-1}+r(\eps)}$ can be approximated by $\Prob{y+U^\eps_{T_0(\eps)}\in A_\eps}$, where $T_0(\eps)$ is a deterministic time, $U^\eps$ is a Gaussian-like process, and $A_\eps\subset \R^d$ is a deterministic set. Namely, the probability of the exit event can be approximated by integrating over $A_\eps$ with respect to a Gaussian-like density. This result is summarized in Lemma \ref{Lemma:rough_approximation} below. The method is to find an explicit expression of $\tau$ by using the fact that $\mathfrak{R}'$, the set to exit, is a box.

In the second step, we show that $\Prob{y+U^\eps_{T_0(\eps)}\in A_\eps}$ is approximately $\Prob{y+\mathcal{Z}\in A_\eps}$ for a centered Gaussian vector $\mathcal{Z}$ with covariance \eqref{eq:def_C_0_Czero}. This is the content of Lemma~\ref{lemma:iteration}. To show this, we apply tools from the Malliavin Calculus to deduce that the density of $U^\eps_{T(\eps)}$ is close to that of $Z_{T(\eps)}$. Here $T(\eps)$ is another deterministic time, which can be much smaller than $T_0(\eps)$, and $Z$ is a Gaussian process independent of~$\eps$ with $Z_\infty $ equal to $\mathcal{Z}$ in distribution. This is done in Lemma \ref{Lemma:density_est}.  
We use an iteration scheme to  extend the Gaussian approximation to the larger time $T_0(\eps)$.

To conclude the proof of this proposition, we estimate the discrepancy between properly scaled $\Prob{y+\mathcal{Z}\in A_\eps}$ and $\psi(f^{-1}(\eps y)/\eps)$, as $\eps \to 0$. This is done in this section after stating Lemma \ref{Lemma:rough_approximation} and Lemma \ref{lemma:iteration}.

\bigskip

To state the two key lemmas, we start by introducing some useful objects.

Since $F(0)=\sigma(0)$ is $d\times n$ with full rank and $F$ is continuous, we can choose $L_0$ so small that there is $c_0>0$ such that $\min_{|u|=1,u\in\R^d}|u^\intercal F(x)|^2\geq c_0$ for all $x\in [-L_0,L_0]^d$, where $\intercal$ stands for matrix transpose. Since we only care about exiting from a subset of $[-L_0,L_0]^d$, we modify $F,G$ outside $[-L_0,L_0]^d$ so that
\begin{align}\label{eq:modified_F}
\begin{split}
    \min_{|u|=1,u\in\R^d}|u^\intercal F(x)|^2\geq c_0, \text{ for all }x\in \R^d; \\
    F, G \text{ and their derivatives are bounded}.
    \end{split}
\end{align}

From now on, we fix this $L_0$ and $F,G$ modified according to~\eqref{eq:modified_F}. By Duhamel's principle, we can solve (\ref{eq:Y_SDE_before_Duhamel}) with $Y_0 = \eps y$ by 
\begin{equation}
\label{eq:Y_after_Duhamel}
Y^j_t = \eps e^{\lj t}(y^j+ U^j_t ),
\end{equation}
where
\begin{equation}
\label{eq:U}
U^j_t=M^j_t + \eps V^j_t
\end{equation}
and
\begin{align}
\label{eq:M}
M^j_t &= \int_0^t e^{-\lj s}F^j_l(Y_s)dW^l_s,\\
\label{eq:V}
V^j_t &= \int^t_0e^{-\lj s}G^j(Y_s)ds.
\end{align}
We emphasize that $M_t$, $V_t$, and $U_t$ depend on $y$ and $\eps$.

To make the notation less heavy we will assume that
\begin{align}\label{eq:assumption_on_R=[-L,L]^d}
    \mathfrak{R}'= [-L,L]^d \quad \text{for some } L\in(0,L_0),
\end{align}
as it is easy to see that for general rectangles, all our arguments still hold.

\begin{Lemma}\label{Lemma:rough_approximation}
    Let 
\begin{equation}    
    T_0 =T_0(\eps)=\alpha \log \eps^{-1}+r(\eps).
\label{eq:T_0}
\end{equation}    
For each $\nu>0$, there are $\eps_0>0$ and $\gj$, $j=1,\ldots,d,$ 
satisfying    
\begin{align}\label{eq:condition_gamma_j}
    (\lj \alpha -1)\vee0<\gj < \lj \alpha, \quad j=1,\ldots,d,
\end{align}    
    such that the following holds for all $y$ satisfying $|y|\leq K'(\eps)$ and all $\eps \leq \eps_0$:
    \begin{align*}
    \begin{split}
        -\eps^\nu+\Prob{y+U_{T_0}\in A_-}\leq \Prob{\tau >\alpha \log\eps^{-1}+r(\eps)}\leq \Prob{y+U_{T_0}\in A_+}+\eps^\nu,
    \end{split}
\end{align*}
where 
\begin{align}\label{eq:def_A_pm}
    \begin{split}
        A_\pm = \{ x\in \R^d:\ |x^j|<\eps^{\lj \alpha -1}Le^{-\lj r(\eps)}\pm \eps^\gj, \quad j=1,\ldots,d \}.
    \end{split}
\end{align}
\end{Lemma}

\begin{Lemma}\label{lemma:iteration}  Let $T_0$ be defined 
 in~\eqref{eq:T_0}  and $\mathcal{Z}$ be a centered Gaussian vector with covariance matrix given by \eqref{eq:def_C_0_Czero}. Then for each $\upsilon \in(0,1)$, there are constants $\eps_0, \delta>0$ such that, for $\eps\in(0,\eps_0]$
 \begin{align*}
    \sup_{|y|\leq \eps^{\upsilon-1}} \big|\Prob{y+U_{T_0}\in A_\pm}-\Prob{y+\mathcal{Z}\in A_\pm}\big| = \smallo{\eps^{\beta(\alpha)+\delta}}.
 \end{align*}
\end{Lemma}

\medskip

These two lemmas are proved in Section \ref{section:approximations}.

\begin{proof}[Proof of Proposition \ref{Prop:Y}]

Let 
\begin{align}\label{eq:def_psi_eps}
        \psi_\eps(y) &= \eps^{-\beta(\alpha)}\Prob{y+\mathcal{Z} \in A_\pm}=\frac{\eps^{-\beta(\alpha)}}{\sqrt{(2\pi)^d\det \Czero}}\int_{A_\pm-y}e^{-\frac{1}{2}x^\intercal \Czero^{-1}x}dx.
\end{align}
Here and below, we use the same argument to treat the cases of $A_+$ and $A_-$ and often omit
the dependence on the choice of $+$ or $-$.

Since we have assumed \eqref{eq:assumption_on_R=[-L,L]^d}, we have
\begin{align*}
    \psi(y)= 
    \begin{cases}
        \frac{\prod_{j< i}2Le^{-\lj r(0)}}{\sqrt{(2\pi)^d \det \Czero}}\int_{\R^{d-i+1}}e^{-\frac{1}{2}x^\intercal\Czero^{-1}x}\big|_{x^{<i}=-y^{<i}}dx^{\geq i},\quad \text{ if }\alpha<\frac{1}{\li},\\
        \frac{\prod_{j< i}2Le^{-\lj r(0)}}{\sqrt{(2\pi)^d \det \Czero}}\int_{(e^{-\li r(0)}[-L,L]-y^i)\times\R^{d-i}}e^{-\frac{1}{2}x^\intercal\Czero^{-1}x}\big|_{x^{<i}=-y^{<i}}dx^{\geq i}, \text{ if }\alpha=\frac{1}{\li}.
    \end{cases}
\end{align*}

The key estimate is the following, to be proved later:
\begin{align}\label{eq:key_psi_eps_and_psi}
\sup_{|y|\leq K'(\eps)}\big|\psi_\eps(y)-\psi\big(\tfrac{f^{-1}(\eps y)}{\eps} \big)\big|\leq \smallo{\eps^q}\text{, for some }q\in (0,1).
\end{align} 

By \eqref{eq:def_psi_eps}, \eqref{eq:key_psi_eps_and_psi}, Lemma \ref{Lemma:rough_approximation} and Lemma \ref{lemma:iteration}, we obtain \eqref{eq:main_Y}. By the discussion above \eqref{eq:main_Y}, the desired result \eqref{eq:main} is attained.
\end{proof}
 
\begin{proof}[Proof of \eqref{eq:key_psi_eps_and_psi}]
We remind the notations introduced in \eqref{eq:projection_notation1}.
Let $\Pi^i$, $\Pi^{>i}$ and $\Pi^{\geq i}$ be projection maps  defined by $\Pi^ix=x^i$, $\Pi^{>i}x = x^{>i}$ and $\Pi^{\geq i}x=x^{\geq i}$. For any set $E\subset \R^d$, we define
         \begin{align}\label{eq:projection_notation2}
        \begin{split}
            &E^{i}=\Pi^iE,\quad E^{> i}=\Pi^{> i}E,\quad 
            E^{\geq i}=\Pi^{\geq i}E.
        \end{split}
    \end{align}
In addition, for a fixed $y\in \R^d$ in \eqref{eq:key_psi_eps_and_psi}, and each $x\in \R^d$, let $\tilde{x}=(-y^{<i},x^{\geq i})\in \R^d $.

Since $\sigma(0)$ has full rank, by the definition of $\Czero$ in \eqref{eq:def_C_0_Czero}, there is $c>0$ such that
\begin{align}\label{eq:boundedness_of_gaussian_density_Czero}
    e^{-\frac{1}{2}x^\intercal\Czero^{-1} x}\leq e^{-c|x|^2}.
\end{align}
Here and below the value of the constant $C$ may vary from instance to instance.
To estimate $|\psi_\eps(y)-\psi\big(\tfrac{f^{-1}(\eps y)}{\eps} \big)|$, we need the following intermediate quantities:
\begin{align*}\begin{split}
        &\mathtt{I}= \frac{\eps^{-\beta(\alpha)}}{\sqrt{(2\pi)^d\det \Czero}}\int_{A_\pm-y}e^{-\frac{1}{2}\tilde{x}^\intercal \Czero^{-1}\tilde{x}}dx, \\
        &\mathtt{II}=\frac{\prod_{j<i}2Le^{-\lj r(0)}}{\sqrt{(2\pi)^d \det \Czero}}\int_{(A_\pm-y)^{\geq i}} e^{-\frac{1}{2}\tilde{x}^\intercal \Czero^{-1}\tilde{x}}dx^{\geq i}.\\
    \end{split}
\end{align*}
Let us write 
\begin{align}\label{eq:split_difference_between_psi_eps-psi}
    \begin{split}
        &\quad\quad\big|\psi_\eps(y)-\psi\big(\tfrac{f^{-1}(\eps y)}{\eps} \big)\big| \leq \big|\psi_\eps(y)-\mathtt{I}\big|+\big|\mathtt{I}-\mathtt{II}\big|+\big|\mathtt{II}-\psi(y)\big|+\big|\psi(y)-\psi\big(\tfrac{f^{-1}(\eps y)}{\eps} \big)\big|,
    \end{split}
\end{align}
and estimate each term on the right of \eqref{eq:split_difference_between_psi_eps-psi}.

By the symmetry and positive definiteness of $\Czero$, we have, for any $x, w \in \R^d$,
\begin{align}\label{eq:difference_gaussian_density}
    \begin{split}
        &\big| e^{-\frac{1}{2}x^\intercal\Czero^{-1}x}-e^{-\frac{1}{2}w^\intercal\Czero^{-1}w}\big| \leq C (e^{-c|x|^2}\vee e^{-c|w|^2})|x+w||x-w|\\
    & \leq C e^{-c|x|^2}(|2x|+|x-w|)|x-w|\Ind{|x|\leq |w|}+C e^{-c|w|^2}(|2w|+|x-w|)|x-w|\Ind{|x|> |w|}\\
    & \leq C(e^{-c_1|x|^2}\vee e^{-c_1|w|^2})\big(|x-w|+|x-w|^2 \big)
    \end{split}
\end{align} 
for some positive $c_1<c$.
Therefore, we have 
\begin{align*}
    \begin{split}
        \sup_{x\in A_\pm -y}&\big| e^{-\frac{1}{2}x^\intercal\Czero^{-1}x}-e^{-\frac{1}{2}\tilde{x}^\intercal\Czero^{-1}\tilde{x}}\big|  \leq \sup_{x\in A_\pm -y}Ce^{-c_1|x^{\geq i}|^2}\big(|x^{<i}+y^{<i}|+|x^{<i}+y^{<i}|^2\big)\\
        &\leq Ce^{-c_1|x^{\geq i}|^2}\sum_{j<i}\Big((\eps^{\lj \alpha-1}Le^{-\lj r(\eps)}+ \eps^\gj )+(\eps^{\lj \alpha-1}Le^{-\lj r(\eps)}+ \eps^\gj )^2\Big)\\
        &\leq C e^{-c_1|x^{\geq i}|^2}\eps^{q_1}
    \end{split}
\end{align*}
for some $q_1>0$. With this, we estimate
\begin{align}\label{eq:est_1st_difference}
    |\psi_\eps(y)-\mathtt{I}|\leq C \eps^{-\beta(\alpha)}\int_{A_\pm -y}e^{-c_1|x^{\geq i}|^2}\eps^{q_1} dx \leq C \eps^{q_1}.
\end{align}
Note that
\begin{align*}
    \mathtt{I} = \frac{\prod_{j<i}2(Le^{-\lj r(\eps)}\pm \eps^{\gj -(\lj \alpha-1)})}{\prod_{j<i}2Le^{-\lj r(0)}}\mathtt{II}.
\end{align*}
Also, clearly we have $|\mathtt{II}|\leq C$. Hence, due to \eqref{eq:condition_for_r(eps)} and \eqref{eq:condition_gamma_j} we have, for some $q_2>0$,
\begin{align}\label{eq:est_2nd_difference}
    \begin{split}
        \big|\mathtt{I}-\mathtt{II}\big|\leq \bigg|\frac{\prod_{j<i}2(Le^{-\lj r(\eps)}\pm \eps^{\gj -(\lj \alpha-1)})}{\prod_{j<i}2Le^{-\lj r(0)}}- 1\bigg|\big|\mathtt{II}\big|\leq C\eps^{q_2}.
    \end{split}
\end{align}

\medskip

For the term $|\mathtt{II}-\psi(y)|$, note that if $i=d+1$, then $\mathtt{II} = \psi(y)$.
Let us consider the case $i\leq d$. Due to \eqref{eq:condition_for_K'},
we have that if either $\alpha<\tfrac{1}{\lambda_i}$ and $j\geq i$, or   $\alpha=\tfrac{1}{\lambda_i}$ and  $j\geq i+1$ , then
\begin{align}\label{eq:pre_est_for_5th_difference}
    \begin{split}
        \int_{\R \setminus (A_\pm -y)^j} e^{-c|x^j|^2}dx^j =& \int_{\eps^{\lj \alpha -1 }Le^{-\lj r(\eps)}\pm \eps^\gj -y^j}^\infty e^{-c|x^j|^2}dx^j\\
        &+\int^{-\eps^{\lj \alpha -1 }Le^{-\lj r(\eps)}\mp \eps^\gj -y^j}_{-\infty}
    e^{-c|x^j|^2}dx^j\\
    \leq& 2\int_{\eps^{\lj \alpha -1}Le^{-\lj r(\eps)}\mp \eps^\gj -K'(\eps)}^\infty e^{-c|x^j|^2}dx^j\leq C\eps^{q_2}.  
    \end{split}
\end{align} 
For the case with $\alpha<\frac{1}{\li}$, by \eqref{eq:boundedness_of_gaussian_density_Czero} and \eqref{eq:pre_est_for_5th_difference}, we have
\begin{align}\label{eq:est_3rd_difference_case_1}
    \begin{split}
        |\mathtt{II}-\psi(y)|&\leq C\int_{\R^{d-i+1}\setminus(A_\pm-y)^{\geq i}}e^{-c|x^{\geq i}|^2}dx^{\geq i}\\
        &\leq C\sum_{j\geq i}\int_{\R \setminus (A_\pm -y)^j} e^{-c|x^j|^2}dx^j\leq C\eps^{q_2}.
    \end{split}
\end{align}
The case with $\alpha=\frac{1}{\li}$ is more involved. Let
\begin{align*}
    \mathtt{III} = \frac{\prod_{j<i}2Le^{-\lj r(0)}}{\sqrt{(2\pi)^d \det \Czero}}\int_{[-Le^{-\li r(0)}-y,Le^{-\li r(0)}-y]\times(A_\pm-y)^{> i}} e^{-\frac{1}{2}\tilde{x}^\intercal \Ceps^{-1}\tilde{x}}dx^{\geq i}.
\end{align*}
Then observe that, with $\triangle$ denoting the symmetric difference of two sets, by \eqref{eq:condition_for_r(eps)},
\begin{align*}
    |\mathtt{II}-\mathtt{III}|&\leq C\int_{(A_\pm-y)^{\geq i}\triangle([-Le^{-\li r(0)}-y,Le^{-\li r(0)}-y]\times(A_\pm-y)^{> i})} e^{-c|x^{\geq i}|^2}dx^{\geq i}\\
    &\leq C\int_{(A_\pm-y)^{i}\triangle[-Le^{-\li r(0)}-y,Le^{-\li r(0)}-y]} e^{-c|x^{i}|^2}dx^{i}\\
    &\leq C(|Le^{-\li r(\eps)}-Le^{-\li r(0)}|+\eps^\gi)\leq C\eps^{q_3}
\end{align*}for some $q_3>0$. On the other hand, by \eqref{eq:pre_est_for_5th_difference}, we have
\begin{align*}
    |\mathtt{III}-\psi(y)|&\leq C\int_{[-Le^{-\li r(0)}-y,Le^{-\li r(0)}-y]\times(\R^{d-i}\setminus(A_\pm-y)^{> i})} e^{-c|x^{\geq i}|^2}dx^{\geq i}\\
    &\leq C\sum_{j> i}\int_{\R \setminus (A_\pm -y)^j} e^{-c|x^j|^2}dx^j\leq C\eps^{q_2}.
\end{align*}
The last two displays together give 
\begin{align}\label{eq:est_3rd_difference_case_2}
    \text{if $\alpha=\tfrac{1}{\li}$,\quad then}\  |\mathtt{II}-\psi(y)|\leq C(\eps^{q_2}+\eps^{q_3}).
\end{align}

\medskip

To estimate the last term $\big|\psi(y)-\psi\big(\tfrac{f^{-1}(\eps y)}{\eps} \big)\big|$, first observe that by \eqref{eq:condition_for_K'}, there exists $\eps_0$ such that for all $\eps\leq \eps_0$, if $|y|\leq K'(\eps)$, then $\eps y\in f(O)$. Due to~\eqref{eq:diffeo-norm}, there is $C>0$ such that $\big|\tfrac{f^{-1}(\eps y)}{\eps}-y\big|\leq C\eps |y|^2$ for all $|y|\leq K'(\eps)$ with $\eps \leq \eps_0$. By this and  \eqref{eq:difference_gaussian_density}, we have, using the exponential term to absorb powers of $|y|$,
\begin{align}\label{eq:est_4th_difference}
    \begin{split}
        \big|\psi(y)-\psi\big(\tfrac{f^{-1}(\eps y)}{\eps} \big)\big| \leq & C\int_{\R^{d-i+1}}e^{-c_1|\tilde{x}|^2}\Big(\big|\tfrac{f^{-1}(\eps y)}{\eps}-y\big|+\big|\tfrac{f^{-1}(\eps y)}{\eps}-y\big|^2\Big)dx^{\geq i}\\
    \leq & C\int_{\R^{d-i+1}}e^{-c_2|\tilde{x}|^2}(\eps +\eps^2)dx^{\geq i} \leq C\eps.\\
    \end{split}
\end{align}

Combining \eqref{eq:split_difference_between_psi_eps-psi},  \eqref{eq:est_1st_difference}, \eqref{eq:est_2nd_difference}, \eqref{eq:est_3rd_difference_case_1}, \eqref{eq:est_3rd_difference_case_2}, and \eqref{eq:est_4th_difference}, we obtain~ \eqref{eq:key_psi_eps_and_psi}.
\end{proof}

\section{Approximations} \label{section:approximations}

\subsection{Proof of Lemma \ref{Lemma:rough_approximation}}
Let us recall that~$\nu>0$ is fixed and we work with processes defined in~\eqref{eq:Y_after_Duhamel}--\eqref{eq:V}. We define an exit time along each direction:
\begin{align}\label{eq:def_tau_j}
    \tj =\inf\{t>0:|Y^j_t|\geq L \}, \quad j=1,2,..., d.
\end{align}
Recalling~\eqref{eq:tau} and \eqref{eq:assumption_on_R=[-L,L]^d},  we obtain $\tau = \min_{1\leq j\leq d}\tj$. By \eqref{eq:condition_for_K'}, there is $\eps_0$ such that for $\eps <\eps_0$, we have $|Y^j_0|=|\eps y^j|\leq L $ for all $j$ and all $y$ with $|y|\leq K'(\eps)$. This fact together with~(\ref{eq:Y_after_Duhamel}) and (\ref{eq:def_tau_j}) implies that, for $\eps <\eps_0$ and $|y|\leq K'(\eps)$,
\begin{align}\label{eq:relation_tj_Uj}
    L=\eps e^{\lj \tj}|y^j+U^j_\tj|\text{,\quad i.e., \ }\tj = \frac{1}{\lj}\log \frac{L}{\eps |y^j+U^j_\tj|}.
\end{align}
Due to~\eqref{eq:T_0}, on $\{\tau >\alpha\log\eps^{-1}+r(\eps)\}$, we have $\tj >T_0$, so~(\ref{eq:relation_tj_Uj}) implies
\begin{align}\label{eq:tau>alpha_log_eps^-1_Probability}
    \begin{split}
        \Prob{\tau >\alpha \log\eps^{-1}+r(\eps)}=\Prob{|y^j + U^j_\tj|<\eps^{\lj \alpha -1}Le^{-\lj r(\eps)},\ j=1,\ldots,d} \\
        = \Prob{|y^j + U^j_{\tj\vee T_0}|<\eps^{\lj \alpha -1}Le^{-\lj r(\eps)} \textrm{\ and\ } \ \tau_j >T_0,\ j=1,\ldots,d}.
    \end{split}
\end{align}

Next, we approximate $U^j_{\tj\vee T_0}$ by $U^j_{T_0}$. Using the definition of $M_t$ given in (\ref{eq:M}) and the boundedness of $F$ and $r(\eps)$, we get, for some $C_1, C_2>0$,
\begin{align*}
    \langle M^j\rangle_{\tj \vee T_0}- \langle M^j\rangle_{ T_0}\leq C_1 e^{-2\lj T_0} \le C_2\eps^{2\lj \alpha}.
\end{align*}
By the exponential martingale inequality (see  \cite[Problem 12.10]{Bass}), this leads to
\begin{align*}
    \Prob{|M^j_{\tj\vee T_0}-M^j_{T_0}|>\tfrac{1}{2}\eps^\gj}\leq 2\exp(-\tfrac{1}{8C_2}\eps^{2\gj-2\lj\alpha}),
\end{align*}
where $\gj$ is chosen to satisfy \eqref{eq:condition_gamma_j}. For the drift term $V_t$, by the boundedness of $G$ and $r(\eps)$, we have the following estimate: for each $q>0$, there is $C_q>0$ such that
\begin{align*}
    \Prob{|\eps V^j_{\tj\vee T_0}-\eps V^j_{T_0}|>\tfrac{1}{2}\eps^\gj}\leq (2\eps^{1-\gj})^q \E{|V^j_{\tj\vee T_0}- V^j_{T_0}|^q}\leq  C_q\eps^{(1-\gj +\lj \alpha)q}.
\end{align*}
By choosing $q$ large, we derive from the above two displays and \eqref{eq:condition_gamma_j} that
\begin{align}\label{eq:U^j_approx}
    \Prob{|U^j_{\tj \vee T_0}-U^j_{T_0}|>\eps^\gj}\leq \eps^\nu,
\end{align}
uniformly in $y$ for $\eps $ small.

Now \eqref{eq:tau>alpha_log_eps^-1_Probability} and \eqref{eq:U^j_approx} immediately imply the upper bound in Lemma \ref{Lemma:rough_approximation}.

\medskip
To get the lower bound, first observe that by \eqref{eq:tau>alpha_log_eps^-1_Probability} we have
\begin{align}
\label{eq:lower_bound_tail}
    \begin{split}
        \Prob{\tau>\log \eps^{-1}+r(\eps)}&\geq \Prob{|y^j + U^j_\tj|<\eps^{\lj \alpha -1}Le^{-\lj r(\eps)},\forall j;\quad |U^j_{\tj}-U^j_{T_0}|\leq\eps^\gj,\forall j} \\
        & \geq \Prob{y+U_{T_0}\in A_-;\quad |U^j_{\tj}-U^j_{T_0}|\leq\eps^\gj,\forall j}\\
        &\geq\Prob{y+U_{T_0}\in A_-}-\Prob{y+U_{T_0}\in A_-; |U^j_{\tj}-U^j_{T_0}|>\eps^\gj,\exists j}.
    \end{split}
\end{align}
To estimate the second term on the right-hand side, we bound it by
\begin{equation}
\label{eq:second-term-lowerbound}
 \Prob{\tau \geq T_0; |U^j_{\tj}-U^j_{T_0}|>\eps^\gj,\exists j }  + \Prob{\tau < T_0; y+U_{T_0}\in A_-}.
\end{equation}
By \eqref{eq:U^j_approx}, the first term can be bounded by $d\eps^\nu$ for $\eps$ small. For the second term, we first introduce the following notations.
For  $x\in \R^d$, $A\subset \R^d$, and $t\in \R$, we write \begin{align}\label{eq:multi_d_notation}
    e^{\lam t}x = (e^{\lj t}x^j)_{j=1}^d\in\R^d, \quad e^{\lam t}A=\{e^{\lam t}x: x\in A \}\subset \R^d.
\end{align}
We recall that if $Y_0=\eps y$, then ~\eqref{eq:Y_after_Duhamel} holds.
Using the strong Markov property of $Y_t$ and the definition of $A_-$ given in \eqref{eq:def_A_pm}, we  obtain
\begin{align}
\label{eq:first-term-lowerbound1}
    \begin{split}
        &\Prob{\tau < T_0; y+U_{T_0}\in A_-}=\Prob{\tau < T_0; Y_{T_0}\in \eps e^{\lambda T_0} A_-}\\
        &\leq \sum_{j=1}^d \Prob{\tau_j<T_0; |Y^j_{T_0}|\leq L-\eps^{1-\lj\alpha +\gj}e^{\lj r(\eps)}}\\
        &\leq \sum_{j=1}^d\EE \ProbxBig{Y_{\tau_j}}{\inf_{t\in[0,T_0]}|Y_t^j|\leq L-\eps^{1-\lj\alpha +\gj}e^{\lj r(\eps)}},
    \end{split}
\end{align}
where $\mathbb{P}^{\mathbf{y}}$ denotes the probability measure under which $Y_t$ satisfies \eqref{eq:Y_SDE_before_Duhamel} with $Y_0=\mathbf{y}\in \R^d$. 
Note that if $|Y^j_0|=L$, then $|Y^j_t|=|e^{\lj t}(Y^j_0+\eps U^j_t)|\geq L-\eps|U^j_t|$. 
By this, using $-\lj\alpha+\gj<0$ (which is due to  \eqref{eq:condition_gamma_j}), the boundedness of $V_t$, and the exponential martingale inequality, we have, for some $c,\ c'>0$ and small $\eps$,
\begin{align}
\label{eq:first-term-lowerbound2}
    \begin{split}
        &\ProbxBig{Y_{\tau_j}}{\inf_{t\in[0,T_0]}|Y_t^j|\leq L-\eps^{1-\lj\alpha +\gj}e^{\lj r(\eps)}}\\
        &\leq \ProbxBig{Y_{\tau_j}}{\inf_{t\in[0,T_0]}(L-\eps|U^j_t|)\leq L-\eps^{1-\lj\alpha +\gj}e^{\lj r(\eps)}} \leq\ProbxBig{Y_{\tau_j}}{\eps^{-\lj\alpha +\gj}e^{\lj r(\eps)}\leq \sup_{t\in[0,T_0]}|U^j_t|}\\
        &\leq \ProbxBig{Y_{\tau_j}}{\eps^{-\lj\alpha +\gj}e^{\lj r(\eps)}-c\eps\leq \sup_{t\in[0,T_0]}|M^j_t|}\leq 2\exp(-c'\eps^{2(-\lj\alpha +\gj)})\leq \eps^\nu.
    \end{split}
\end{align}
Combining \eqref{eq:lower_bound_tail}, \eqref{eq:second-term-lowerbound}, \eqref{eq:first-term-lowerbound1}, and \eqref{eq:first-term-lowerbound2} leads to the desired lower bound.

\subsection{Proof of Lemma \ref{lemma:iteration}}
First of all, we state two density estimates that we need. For a random variable $\mathcal{X}$ with values in $\R^d$, its Lebesgue density, if exists, is denoted by $\rho_\mathcal{X}$. Since $U_t$ in \eqref{eq:Y_after_Duhamel} depends on $y$, we denote its density by $\rho^y_{U_t}$. 
\begin{Lemma}\label{Lemma:density_est}
     Consider \eqref{eq:Y_after_Duhamel} with $Y_0=\eps y$. Let  
     \begin{align}\label{eq:def_pp}
         \pp(x)=\sum_{j,k=1}^dx^\frac{\lj}{\lk},\quad \text{for }x\geq 0,
     \end{align}
    \begin{align}\label{eq:def_Z}
    Z^j_t = \int_0^t e^{-\lj s}F^j_l(0)dW^l_s.
    \end{align}
   Then
    \begin{enumerate}[1)]
        \item there is $\theta>0$ such that for each $\upsilon \in (0,1)$ there are $C,c,\delta>0$ such that, for $\eps$ sufficiently small,
    \begin{align*}
        |\rho^y_{U_{T(\eps)}}(x)-\rho_{Z_{T(\eps)}}(x)|\leq C\eps^\delta\big(1+\pp(\eps^{1-\upsilon}|y|)\big) e^{-c|x|^2},\quad  x, y\in \R^d, 
    \end{align*}
        holds for all deterministic functions $T(\cdot)$ satisfying $1\leq T(\eps)\leq \theta\log\eps^{-1}$, $\eps\in(0,1)$;
    \item for each $\theta'>0$, there are $C',c',\delta'>0$ such that, for $\eps$ sufficiently small,
    \begin{align*}
          |\rho_{Z_{T(\eps)}}(x)-\rho_{Z_\infty}(x)|\leq C'\eps^{\delta'} e^{-c'|x|^2},\quad x\in \R^d, 
    \end{align*}
        holds for all deterministic functions $T(\cdot)$ satisfying $T(\eps)\geq \theta'\log\eps^{-1}$, 
        $\eps\in(0,1)$.
    \end{enumerate}
\end{Lemma}
This lemma will be proved in Section \ref{section:density_estimate}. 

We recall the notation introduced in \ref{eq:multi_d_notation} and the definition of $T_0=T_0(\eps)$ in~\eqref{eq:T_0}. We set $N=\min\{n\in\mathbb{N}:\frac{T_0}{n}\leq \theta\log\eps^{-1}, \forall \eps \in(0,1/2]\}$, where $\theta$ was introduced in Lemma \ref{Lemma:density_est},
 and $t_k=\frac{k}{N}T_0$.  Hence, each increment $t_k-t_{k-1}$ satisfies the condition imposed on time $T(\eps)$ in part 1 of Lemma \ref{Lemma:density_est}, so we can get the following iteration result.

\begin{Lemma}\label{lemma:iterative_est}
For each $\upsilon\in(0,1)$, there are constants $\eps_k,C_k,\delta_k>0$, $k=1,2,...,N$, and  $\upsilon'>0$ such that
\begin{align}\label{eq:induction_step_conclusion}
    \sup_{|y|\leq \eps^{\upsilon-1}}\sup_{|w|\leq \eps^{\upsilon'-1}}\big|\Probx{\eps y}{y+U_{t_k}+e^{-\lam t_k}w\in A_\pm} -  \Prob{y+Z_{t_k}+e^{-\lam t_k}w\in A_\pm}\big|\leq C_k \eps^{\beta(\alpha)+\delta_k},
\end{align}
holds for each $k=1,2,..., N$ and for all $\eps\in(0,\eps_k]$.
\end{Lemma}
Let us first derive  Lemma \ref{lemma:iteration} from Lemmas~\ref{Lemma:density_est} and~\ref{lemma:iterative_est}, and then return to the proof  of the latter.

\begin{proof}[Proof of Lemma \ref{lemma:iteration}]
Set $k=N$ and $w=0$ in Lemma \ref{lemma:iterative_est}. As $t_N=T_0$, we have that for each $\upsilon\in(0,1)$, there is $\delta>0$ such that
\begin{align*}
    \sup_{|y|\leq \eps^{\upsilon-1}}\big|\Probx{\eps y}{y+U_{T_0}\in A_\pm} -  \Prob{y+Z_{T_0}\in A_\pm}\big|=\smallo{\eps^{\beta(\alpha)+\delta}}.
\end{align*}
It is easy to see from \eqref{eq:def_Z} that $Z_\infty$ is defined (in the sense of a.s.-convergence) and has the same distribution as $\mathcal{Z}$: it is a centered Gaussian vector with covariance matrix \eqref{eq:def_C_0_Czero} since $F(0)=\sigma(0)$ by~\eqref{eq:F(0)=simga(0)}.  Taking $\theta'>0$ such that $T_0\geq \theta'\log\eps^{-1}$ for all $\eps$, part 2 of Lemma \ref{Lemma:density_est} and the definition of $A_\pm$ given in \eqref{eq:def_A_pm}, imply that, for $\eps$ sufficiently small, 
\begin{align*}
    |\Prob{y+Z_{T_0}\in A_\pm} -  \Prob{y+\mathcal{Z}\in A_\pm}\big|=\smallo{\eps^{\beta(\alpha)+\delta'}},\quad\forall y\in\R^d.
\end{align*}
The above two displays together imply the desired result.
 \end{proof}

\begin{proof}[Proof of Lemma \ref{lemma:iterative_est}]
Let us choose $\upsilon'\in(0,1)$ to satisfy
\begin{align}\label{eq:condition_upsilon'}
    \frac{\lj \alpha}{N} \geq \frac{\lam_d\alpha}{N} >\upsilon', \quad \text{ for all }j=1,2,...,d. 
\end{align}

For the case $k=1$, \eqref{eq:induction_step_conclusion} follows from Lemma \ref{Lemma:density_est} and the definition of $A_\pm$ in \eqref{eq:def_A_pm}. Then we proceed by induction. Assume \eqref{eq:induction_step_conclusion} holds for $k-1$ with $k\leq N$. 

Set $z(u)= e^{\lam t_{k-1}}(y+u)$. The strong Markov property of $Y_t$ implies that 
\begin{align}\label{eq:strong_markov_Y}
    \begin{split}
        \Probx{\eps y}{y+U_{t_k}+e^{-\lam t_k}w\in A_\pm}  = & \Probx{\eps y}{Y_{t_k}+\eps w\in \eps e^{\lam t_k}A_\pm} \\ = &  \EX{\eps y}{\Probx{ Y_{t_{k-1}}}{Y_{t_1}+\eps w\in \eps e^{\lam t_k}A_\pm}} \\
         = & \EXBig{\eps y}{\Probx{ \eps z(u)}{z(u)+U_{t_1}+e^{-\lam t_1}w\in e^{\lam t_{k-1}}A_\pm}|_{u=U_{t_{k-1}}}}.
    \end{split}
\end{align}
We will show the error of replacing $U_{t_1}$ by $Z_{t_1}$ and $U_{t_{k-1}}$ by $Z_{t_{k-1}}$ is small.

Using Lemma \ref{Lemma:density_est} (1) with $\upsilon'$ in place of $\upsilon$, we see that  there are $\delta', C',c'$ such that 
\begin{align*}
    \begin{split}
        &\big|\Probx{ \eps z(u)}{z(u)+U_{t_1}+e^{-\lam t_1}w\in e^{\lam t_{k-1}}A_\pm}- \Probx{ \eps z(u)}{z(u)+Z_{t_1}+e^{-\lam t_1}w\in e^{\lam t_{k-1}}A_\pm} \big|\\
        &\leq \int_{\{x\in\R^d:z(u)+x+e^{-\lam t_1}w\in e^{\lam t_{k-1}}A_\pm\}}C'\eps^{\delta'}\big(1+\pp(\eps^{1-\upsilon'}|z(u)|)\big)e^{-c'|x|^2}dx.
    \end{split}
\end{align*}
By \eqref{eq:condition_upsilon'},  $t_{k-1}=\frac{k-1}{N}T_0$, and $k\leq N$, we have
\begin{align*}
    e^{\lj t_{k-1}}\eps^{\lj \alpha -1}\leq e^{\lj t_{N-1}}\eps^{\lj \alpha -1} \leq \eps^{\frac{1}{N}\lj \alpha -1 }e^{\frac{N-1}{N} r(\eps)} < \eps^{\upsilon'-1}e^{\frac{N-1}{N}r(\eps)}.
\end{align*}
Together with the definition of $A_\pm$ in \eqref{eq:def_A_pm}, this implies that, for some $C>0$, 
\begin{align*}
    \eps^{1-\upsilon'}|z(u)| \leq C+\eps^{1-\upsilon'}|x|,
\end{align*}
for $z(u)$ satisfying $z(u)+x+e^{-\lam t_1}w\in e^{\lam t_{k-1}}A_\pm$ and $|w|\leq \eps^{\upsilon'-1}$.
Using $e^{-c'|x|^2}$ to absorb powers of $|x|$, the above three displays give, for some $C,c>0$,
\begin{align*}
    \begin{split}
        &\big|\Probx{ \eps z(u)}{z(u)+U_{t_1}+e^{-\lam t_1}w\in e^{\lam t_{k-1}}A_\pm}- \Prob{z(u)+Z_{t_1}+e^{-\lam t_1}w\in e^{\lam t_{k-1}}A_\pm} \big|\\
        &\leq \eps^{\delta'}\int_{\{x\in\R^d:z(u)+x+e^{-\lam t_1}w\in e^{\lam t_{k-1}}A_\pm\}}Ce^{-c|x|^2}dx, \quad|w|\leq \eps^{\upsilon'-1}.
    \end{split}
\end{align*}
Let $\NN$ be a centered Gaussian with density proportional to $e^{-c|x|^2}$ and independent of $\mathcal{F}_{t_{k-1}}$. Then the above display and \eqref{eq:strong_markov_Y} imply that
\begin{align*}
    \begin{split}
        \mathtt{I}&= \Big|\Probx{\eps y}{y+U_{t_k}+e^{-\lam t_k}w\in A_\pm}-\EXBig{\eps y}{\Prob{z(u)+Z_{t_1}+e^{-\lam t_k}w\in e^{\lam t_{k-1}}A_\pm}|_{u=U_{t_{k-1}}}}\Big|\\
        &\leq C\eps^{\delta'}\Probx{\eps y}{y+U_{t_{k-1}}+e^{-\lam t_k}w+e^{-\lam t_{k-1}}\NN \in A_\pm},\quad|w|\leq \eps^{\upsilon'-1}.
    \end{split}
\end{align*}
Then we choose $\rho$ large so that
\begin{align}\label{eq:choice_of_rho_in_iteration}
    \Prob{|\NN|>\rho\log\eps^{-1}}=o(\eps^{\beta(\alpha)}),\quad \Prob{|Z_{t_1}|>\rho\log\eps^{-1}}=o(\eps^{\beta(\alpha)+\delta'}).
\end{align}
Note that $e^{-\lambda t_1}$ decays like a small positive power of $\eps$. So, there is $\eps_k$ such that
\begin{align}\label{eq:induction_assumption_satisfied}
    \text{if }|w| \leq \eps^{\upsilon'-1}\text{ for }\eps \leq \eps_k ,\quad\text{ then }|e^{-\lam t_1}w|+\rho\log\eps^{-1}\leq \eps^{\upsilon'-1} \text{ for }\eps \leq \eps_{k-1}.
\end{align}

Then, the following holds uniformly over $|y|\leq \eps^{\upsilon-1},|w|\leq \eps^{\upsilon'-1}$ and $\eps \in(0,\eps_k]$:
\begin{align*}
\begin{split}
        \mathtt{I}&\leq C\eps^{\delta'}\Probx{\eps y}{y+U_{t_{k-1}}+e^{-\lam t_{k-1}}(e^{-\lam t_1}w+\NN) \in A_\pm; |\NN|\leq \rho\log\eps^{-1}} + \smallo{\eps^{\beta(\alpha)+\delta'}}\\
    &\leq C\eps^{\delta'}\Prob{y+Z_{t_{k-1}}+e^{-\lam t_{k-1}}(e^{-\lam t_1}w+\NN)\in A_\pm}+\smallo{\eps^{\beta(\alpha)+\delta_{k-1}+\delta'}}+\smallo{\eps^{\beta(\alpha)+\delta'}},
\end{split}
\end{align*}
where in the second inequality we used the induction assumption allowed by \eqref{eq:induction_assumption_satisfied}, independence of $\NN$, Fubini's theorem, and \eqref{eq:choice_of_rho_in_iteration}. One can check that for $k-1 \geq 1$, there are $C,c>0$ such that $\rho_{Z_{t_{k-1}}}(x)\leq Ce^{-c|x|^2}$ for all $x\in\R^d$. We also recall that $i=i(\alpha)$ is given in \eqref{eq:def_i=i(alpha)}. Hence, we can estimate, using Fubini's theorem, the definition of $A_\pm$ given in \eqref{eq:def_A_pm}, the definition of $\beta(\alpha)$ in~\eqref{eq:beta_alpha} and notations given in \eqref{eq:projection_notation1}--\eqref{eq:projection_notation2},
\begin{align*}
    &\Prob{y+Z_{t_{k-1}}+e^{-\lam t_{k-1}}(e^{-\lam t_1}w+\NN)\in A_\pm}\leq\E\int_{A_\pm-y-e^{-\lam t_{k-1}}(e^{-\lam t_1}w+\NN)} Ce^{-c|x|^2}dx\\
    &\leq C\int_{(A_\pm)^{<i}\times \R^{d-i}}e^{-c|x^{\geq i}|^2}dx \leq C\big|(A_\pm)^{<i}\big|\leq C'\eps^{\beta(\alpha)}.
\end{align*}
The above two displays indicate that, for some $\delta''>0$
\begin{align}\label{eq:iterative_est_est_of_I}
    \mathtt{I} =\smallo{\eps^{\beta(\alpha)+\delta''}}, \quad\text{uniformly over $|y|\leq \eps^{\upsilon-1}$, $|w|\leq \eps^{\upsilon'-1}$}. 
\end{align}

Then we estimate the error caused by replacing $U_{t_{k-1}}$ by $Z_{t_{k-1}}$. Let $\tZ_{t_1}$ be a copy of $Z_{t_1}$ independent of $\mathcal{F}_{t_{k-1}}$. Using this independence and \eqref{eq:choice_of_rho_in_iteration}, we have that the following holds uniformly over $|y|\leq \eps^{\upsilon-1}$ and $|w|\leq\eps^{\upsilon'-1}$ with $\eps \in(0,\eps_k]$:
\begin{align*} 
    &\EXBig{\eps y}{\Prob{z(u)+Z_{t_1}+e^{-\lam t_1}w\in e^{\lam t_{k-1}}A_\pm}|_{u=U_{t_{k-1}}}}\\
    &= \Probx{\eps y}{y+U_{t_{k-1}}+e^{-\lam t_{k-1}}(e^{-\lam t_1}w+\tZ_{t_1})\in A_\pm}\\
    & = \Probx{\eps y}{y+U_{t_{k-1}}+e^{-\lam t_{k-1}}(e^{-\lam t_1}w+\tZ_{t_1})\in A_\pm; |\tZ_{t_1}|\leq \rho\log\eps^{-1}}+\smallo{\eps^{\beta(\alpha)+\delta'}}\\
    &= \Probx{\eps y}{y+Z_{t_{k-1}}+e^{-\lam t_{k-1}}(e^{-\lam t_1}w+\tZ_{t_1})\in A_\pm; |\tZ_{t_1}|\leq \rho\log\eps^{-1}}+\smallo{\eps^{\beta(\alpha)+\delta_{k-1}}}+\smallo{\eps^{\beta(\alpha)+\delta'}}\\
    &= \Probx{\eps y}{y+Z_{t_{k-1}}+e^{-\lam t_{k-1}}\tZ_{t_1}+e^{-\lam t_k}w\in A_\pm}+\smallo{\eps^{\beta(\alpha)+\delta_{k-1}\wedge \delta'}},
\end{align*}
where we used the induction assumption in the third identity allowed by \eqref{eq:induction_assumption_satisfied}, independence of $\tZ_{t_1}$ and Fubini's theorem.
By this independence again, a simple computation reveals that $Z_{t_{k-1}}+e^{-\lam t_{k-1}}\tZ_{t_1}$ has the same distribution as that of $Z_{t_k}$. Then the above display implies that
\begin{align*}
    &\Big|\EXBig{\eps y}{\Prob{z(u)+Z_{t_1}+e^{-\lam t_1}w\in e^{\lam t_{k-1}}A_\pm}|_{u=U_{t_{k-1}}}}- \Probx{\eps y}{y+Z_{t_k}+e^{-\lam t_k}w\in A_\pm}\Big|\\
    &=\smallo{\eps^{\beta(\alpha)+\delta_{k-1}\wedge \delta'}}, \quad \text{uniformly over $|y|\leq\eps^{\upsilon-1}$, $|w|\leq \eps^{\upsilon'-1}$  and $\eps \in(0,\eps_k]$. }
\end{align*}
From this and \eqref{eq:iterative_est_est_of_I}, we derive \eqref{eq:induction_step_conclusion} for $k$, which completes our proof.
\end{proof}

\section{Density estimate}\label{section:density_estimate}

In this section, we prove Lemma \ref{Lemma:density_est}.

We briefly introduce Malliavin calculus notations. For $\mathcal{T}>0$, on $(W_t, t\in[0,\mathcal{T}])$, let $\D$ be the derivative operator; $\sigma_{\mathcal{X}}$ be the Malliavin covariance matrix for a random vector $\mathcal{X}\in \mathcal{F}_\mathcal{T}$;  $\|\cdot \|_{k,p,\mathcal{T}}$ be the Sobolev norm defined in terms of derivatives up to the $k$th order with $L^p$ integrability; $\mathbb{D}^{k,p}(\mathcal{T})$ be the corresponding Sobolev space, in particular, $\mathbb{D}^{k,\infty}(\mathcal{T}) = \cap_{p\geq1}\mathbb{D}^{k,p}(\mathcal{T})$. More details can be found in \cite{nualart}.

Theorem 2.14.B from \cite{bally2014} estimates the difference between derivatives of two densities in terms of Sobolev norms and the covariance matrix. For our purposes, in our statement of this result, Theorem~\ref{thm:DensityDifference} below, we simplify the conditions of the original theorem by setting the localization random variable~$\mathbf{\Theta}$ to be $1$, the derivative order $q=0$ and using Meyer's inequality (c.f. \cite[Theorem~1.5.1]{nualart}) to bound the Ornstein--Uhlenbeck operator. We stress that, although the conditions 
of  Theorem~2.14.B as it is stated in~\cite{bally2014} do not formally allow for $q=0$, that theorem is still
valid for this value of~$q$. In fact,  in~\cite{bally2014}, Theorem~2.14 is derived from Theorem~2.1 via an approximation argument. In turn, part B of Theorem 2.1 is restated and proved in the form of 
Theorem 3.10, where~$q$ is allowed to be $0$.
\begin{Th} \label{thm:DensityDifference}
For $i=1,2$, let $\mathcal{X}_i \in \mathbb{D}^{3,\infty}(\mathcal{T})$ with values in $\R^d$ satisfy 
$\E{(\det\sigma_{\mathcal{X}_i})^{-p} }<\infty$ for every $p >1$.

Then there exist positive constants $C,a,b, \gamma$ only depending on $d$ such that for all $x\in\R^d$
\begin{align*}
    |\rho_{\mathcal{X}_1}(x)-\rho_{\mathcal{X}_2}(x)| \leq & C \|\mathcal{X}_1-\mathcal{X}_2 \|_{2,\gamma,\mathcal{T}}\bigg(\prod_{i=1,2}\Big(1\vee \E{(\det\sigma_{\mathcal{X}_i})^{-\gamma} }\Big)\Big(1+\| \mathcal{X}_i\|_{3,\gamma,\mathcal{T}}\Big) \bigg)^a\\
    & \cdot \Big(\sum_{i=1,2}\Prob{|\mathcal{X}_i-x|<2}\Big)^b.
\end{align*}
\end{Th}
The independence of $C,a,p$ of $\mathcal{T}$ is important because we will replace $\mathcal{T}$ by a function of $\eps$ converging to $\infty$ as $\eps \to 0$.

Let us fix $\theta$ and $\eps_0$ such that
\begin{align*} 
        2\lone \theta \leq 1,  \quad\text{ and }\quad
       \eps^\frac{1}{2}\theta\log(\eps^{-1})\leq 1, \quad \eps \in(0,\eps_0].
\end{align*}
For all deterministic $T=T(\eps)$ satisfying $1\leq T\leq \theta\log(\eps^{-1})$, we have
\begin{align}\label{eq:choice_of_theta}
    \begin{split}
        \eps e^{2\lj T}\leq\eps e^{2\lone T}\leq 1 \text{ for all }j,\quad\textrm{and}\quad\eps^\frac{1}{2}T\leq 1, \quad\quad \eps \in(0,\eps_0].
    \end{split}
\end{align}
Now, arbitrarily fix such a $T=T(\eps)$. We will use $\mathcal{T}=T(\eps)$ and simply write $\|\cdot \|_{k,p}=\|\cdot \|_{k,p,T(\eps)}$.

In the following, we use $\lesssim$ to omit a positive multiplicative constant independent of~$\eps$ and $T=T(\eps)\in[1,\theta \log (\eps^{-1})]$. Sometimes such a constant will be denoted explicitly but generically as $C$. We also use the bracket $[\cdot]_p= \big(\E{|\cdot|^p}\big)^\frac{2}{p}$ for $p\geq 2$ and note that this bracket satisfies the following properties, by BDG, Cauchy--Schwarz and Minkowski's integral inequalities, for $p \geq 2$: 
\begin{align} \label{eq:pBracket_property_1}
\begin{split}
    &\pbracketBig{\int_{t_1}^{t_2}\mathcal{X}_{l,s}dW^l_s} = \Big(\EBig{\Big|\int_{t_1}^{t_2}\mathcal{X}_{l,s}dW^l_s\Big|^p}\Big)^\frac{2}{p} \lesssim   \Big(\EBig{\Big|\int_{t_1}^{t_2}|\mathcal{X}_{s}|^2 ds\Big|^\frac{p}{2}}\Big)^\frac{2}{p} \\
    &\leq \int_{t_1}^{t_2}\Big(\E{|\mathcal{X}_{s}|^p }\Big)^\frac{2}{p} ds =  \int_{t_1}^{t_2}\pbracket{\mathcal{X}_{s}}ds,
    \end{split}
\end{align}
\begin{align}\label{eq:pBracket_property_2}
    \pbracketBig{\int_{t_1}^{t_2}\mathcal{X}_s ds }=\Big(\EBig{\Big|\int_{t_1}^{t_2}\mathcal{X}_{s}ds\Big|^p}\Big)^\frac{2}{p} \leq \Big(\int_{t_1}^{t_2}
    \big(\E{|\mathcal{X}_{s}|^p}\big)^\frac{1}{p}ds\Big)^2\leq |t_2-t_1|\int_{t_1}^{t_2}\pbracket{\mathcal{X}_s}ds,
\end{align}
\begin{align*}
    \pbracketBig{\Big(\int_{E\subset \R^n}\big|\mathcal{X}_{s^1,s^2,\dots,s^n}\big|^2 ds^1ds^2\dots ds^n\Big)^\frac{1}{2} }\leq \int_{E\subset \R^n}\pbracket{\mathcal{X}_{s^1,s^2,\dots,s^n}} ds^1ds^2\dots ds^n.
\end{align*}
Let $\mathcal{H}=L^2([0,T],\R^d) $. The last property above implies
\begin{align}\label{eq:pBracket_property_3}
    \pbracket{\|\mathcal{X}\|_{\Hil^{\otimes n}}}\leq  \int_{[0,T]^n}\pbracket{\mathcal{X}_{s^1,s^2,\dots,s^n}}ds^1ds^2\dots ds^n, \quad n\in\mathbb{N}\setminus\{0\},
\end{align}
where $\Hil^{\otimes n}$ is the $n$-fold tensor product of $\Hil$ and $\mathcal{X}$ is an $\Hil^{\otimes n}$-valued random variable.

In the following, we fix an arbitrary $p\geq 2$, and use the above properties.

\subsection{Estimates of Malliavin derivatives}
The formulae for Malliavin derivatives of a solution of an SDE can be found in \cite[Section 2.2.2]{nualart}.  We will use them without further notice. 

\begin{Rem}
\label{rem:smooth-coef}
In \cite[Section~2.2.2]{nualart}, the coefficients of the SDE are required to 
be~$C^\infty$ in order to compute Malliavin derivatives of all orders but here we need to work only with  Malliavin derivatives up to order~$3$, and our assumptions on smoothness of the coefficients are sufficient.
\end{Rem}
\smallskip

Let $N_t = U_t - Z_t$ and $H(\cdot) = F(\cdot)-F(0)$. By \eqref{eq:modified_F}, there are $C_{H,1}, C_{H,2}>0$ such that, 
\begin{align} \label{eq:Lipschitz_H}
\begin{split}
    |H(x)|\leq C_{H,1}|x| , \quad
    |H(x)|\leq C_{H,2},\quad \quad x\in \R^d.
    \end{split}
\end{align}
For $0\leq r\leq t\leq T$, by \eqref{eq:Y_after_Duhamel} and \eqref{eq:def_Z}, easy calculations yield
\begin{align}\label{eq:BasicDerivativeCalc}
\begin{split}
       & N^i_t = \int_0^{t}e^{-\li s}H^i_l(Y_s)dW^l_s+\eps V^i_t; \quad \D^j_r Z^i_t = e^{-\li r}F^i_j(0); \\
    & \D^j_r Y^i_t = \eps e^{\li t}(\D^j_r U^i_t ) =\eps  e^{\li t}(\D^j_r N^i_t + \D^j_r Z^i_t) .
\end{split}
\end{align}

\subsubsection{$0$th order derivatives} 
For some $\beta\in (0,1)$ to be chosen later, we define the stopping times $\eta_k = \inf\{t>0:|Y^k_t|\geq \eps^\beta\}$ and $\eta = \min_{1\leq k\leq d}\{\eta_k\}$

Using (\ref{eq:Lipschitz_H}), (\ref{eq:BasicDerivativeCalc}), and the boundedness of $V^i_t$ we have
\begin{align}\label{eq:pre_0th_derivative_estimate}
\begin{split}
    \E{|N^i_T|^p}&\lesssim \EBig{\Big|\int_0^T|e^{-\li s}H^i(Y_s)|^2 ds\Big|^\frac{p}{2}} +\eps^p\\
     &\lesssim \EBig{\Big|\int_0^{T\wedge \eta}|e^{-\li s}H^i(Y_s)|^2 ds\Big|^\frac{p}{2}} + \EBig{\Big|\int_{T\wedge \eta}^T|e^{-\li s}H^i(Y_s)|^2 ds\Big|^\frac{p}{2}}+\eps^p\\
     &\lesssim \sum_{k=1}^d\EBig{\Big|\int_{0}^{T\wedge\eta}|e^{-\li s}Y^k_s|^2 ds\Big|^\frac{p}{2}} +  \EBig{\Big|\int_{T\wedge\eta}^T|e^{-\li s}|^2 ds\Big|^\frac{p}{2}}+\eps^p\\
     &\lesssim \sum_{k=1}^d\EBig{\Big|\int_{0}^{T\wedge\eta}|e^{-\li s}\eps^\beta|^2 ds\Big|^\frac{p}{2}}  + \E{e^{-p\li \eta}}+\eps^p\\
     & \lesssim \eps^{p\beta}+  \E{e^{-p\li \eta}}+\eps^p.
\end{split}
\end{align}

By the definition of $\eta_k$ and the relation $\eps^\beta \leq |Y^k_{\eta_k}|=\eps e^{\lk \eta_k}|y^k+U^k_{\eta_k}|$, we have \\ $\eta_k \geq  \frac{1}{\lk}\log(\eps^{\beta-1}|y^k+U^k_{\eta_k}|^{-1})$, which implies that
\begin{align*}
    \E{e^{-p\li \eta}}&\leq \sum_{k=1}^d\E{e^{-p\li \eta_k}}\leq \sum_{k=1}^d \eps^{(1-\beta)p\frac{\li}{\lk}}\E{|y^k+U^k_{\eta_k}|^{p\frac{\li}{\lk}}}\\
    &\lesssim \sum_{k=1}^d \big(\eps^{(1-\beta)}|y|\big)^{p\frac{\li}{\lk}}+ \sum_{k=1}^d \eps^{(1-\beta)p\frac{\li}{\lk}}\E{|U^k_{\eta_k}|^{p\frac{\li}{\lk}}}.
\end{align*}
Note that any positive moment of $U_{\eta_k}$ is bounded by an absolute constant independent of $\eps$. Recall the definition of $\pp$ given in \eqref{eq:def_pp}. Then, in view of the above display and \eqref{eq:pre_0th_derivative_estimate}, for an arbitrary $\upsilon\in(0,1)$ we can choose $\beta = \frac{1}{2}\upsilon$ so that there is $\delta_0$ independent of $p$ such that
\begin{align}\label{eq:0th_derivative__norm_est}
    \pbracket{N^i_T} \lesssim \eps^{\delta_0}\big(1+\pp(\eps^{1-\upsilon}|y|)\big)^2,\quad i=1,2,\dots, d, \quad \eps \in(0,\eps_0].
\end{align}

\subsubsection{$1$st order derivatives}

Consider $ r\leq t\leq T$. Before estimating $\D^j_r N^i_t$, we first study $\D^j_r U^i_t$. Observe that 
\begin{align}\label{eq:formula_of_1st_derivative_U}
    \D^j_rU^i_t = e^{-\li r}F^i_j(Y_r) + \eps \int_r^te^{(\lk -\li)s}\partial_k F^i_l(Y_s)\D^j_rU^k_sdW^l_s+\eps^2\int_r^t e^{(\lk-\li)s}\partial_k G^i(Y_s)\D^j_r U^k_s ds.
\end{align}
Hence, using the boundedness of the derivatives of $F$ and $G$ due to \eqref{eq:Y_after_Duhamel}, the $[\cdot]_p$ properties \eqref{eq:pBracket_property_1} and \eqref{eq:pBracket_property_2}, and lastly \eqref{eq:choice_of_theta}, we have 
\begin{align*}
    \pbracket{\D^j_r U^i_t} & \lesssim \pbracket{e^{-\li r}F^i_j(Y_r)} + \pbracketBig{\eps \int_r^te^{(\lk -\li)s}\partial_k F^i_l(Y_s)\D^j_rU^k_sdW^l_s}\\
    &+\pbracketBig{\eps^2\int_r^t e^{(\lk-\li)s}\partial_k G^i(Y_s)\D^j_r U^k_s ds}\\
    & \lesssim e^{-2\li r} + (\eps^2+\eps^4 T) \sum_{k=1}^d\int_r^te^{2(\lk -\li)s}\pbracket{\D^j_rU^k_s}ds\\
    &\lesssim e^{-2\li r} + \eps^2 \sum_{k=1}^d\int_r^te^{2(\lk -\li)s}\pbracket{\D^j_rU^k_s}ds. 
\end{align*}
We fix $j, r$ momentarily and set $a^i(t) = \pbracket{\D^j_r U^i_t}$, obtaining a system of inequalities
\begin{align}\label{eq:common_sys_inequalities}
    a^i(t) \lesssim  e^{-2\li r} +  \eps^2 \sum_{k=1}^d\int_r^te^{2(\lk -\li)s}a^k(s)ds, \quad i=1,2,...,d.
\end{align}
This type of systems will occur a few more times. So, it is  useful to state the following bound  proved in Section \ref{section:auxiliary_lemmas}:

\begin{Lemma} \label{lemma:solve_inequality_system}
Let $d\in\N$ and $m\geq 0$. Then the system of inequalities
\begin{align}\label{eq:system_of_ineq}
    0\leq a^i(t) \lesssim  \eps^me^{-2\li r} +  \eps^2 \sum_{k=1}^d\int_r^te^{2(\lk -\li)s}a^k(s)ds, \quad t\in [r,T],\ i=1,2,...,d,
\end{align}
with $T$ and $\eps\in(0,\eps_0]$ satisfying (\ref{eq:choice_of_theta}), implies that there is a constant $C$ independent of $\eps,\ T,$ and $r$ such that $a^i(t)\le C\eps^m e^{-\lam_d r}$ for all
$t\in[r,T]$ and $i=1,2,\dots,d$.
\end{Lemma}

Applying this lemma to (\ref{eq:common_sys_inequalities}), we obtain
\begin{equation}
\label{eq:pbracket-of-DjrUit}
\pbracket{\D^j_r U^i_t} = a^i(t) \le C e^{-2\lam_d r}, \quad r\le t\leq T, \quad p \geq 2, \quad \eps \in(0,\eps_0],
\end{equation}
 which gives, by \eqref{eq:choice_of_theta} and \eqref{eq:pBracket_property_3},
\begin{align}\label{eq:1st_Derivative_U}
    \begin{split}
    \pbracket{\|\D U_T \|_{\Hil}}& \leq \int_0^T\pbracket{\D_rU_T}dr \lesssim \sum_{i,j=1}^d\int_0^T\pbracket{\D^{j}_{r}U^i_T}dr \lesssim  1, \quad p \geq 2.
    \end{split}
\end{align}
The following estimate implied by \eqref{eq:choice_of_theta} and \eqref{eq:pbracket-of-DjrUit} will be used later:
\begin{align}\label{eq:1st_Y_derivative_estimate}
    \begin{split}
        \pbracket{\D^j_r Y^i_t}&=  \eps^2 e^{2\li t}\pbracket{\D^j_r U^i_t} \lesssim \eps^2 e^{2\lam_1 t} e^{-2\lam_d r} \lesssim \eps, \quad r\le t\leq T, \quad p \geq 2, \quad \eps \in(0,\eps_0].
    \end{split}
\end{align}
Then we proceed to estimating $\D^j_rN^i_t$. The calculation (\ref{eq:BasicDerivativeCalc}) gives, for $r\leq t\leq T$,
\begin{align*}
    \D^j_rN^i_t =e^{-\li r}H^i_j(Y_r) +  \int_r^te^{ -\li s}\partial_k H^i_l(Y_s)\D^j_r Y^k_sdW^l_s+\eps\int_r^t e^{-\li s}\partial_k G^i(Y_s)\D^j_r Y^k_s ds,
\end{align*}
which implies
\begin{align*}
    &\pbracket{\|\D N_T\|_\Hil} \lesssim \sum_{i=1}^d \bigg( \EBig{\Big|\int_0^T|e^{-\li r}H^i(Y_r)|^2 dr \Big|^\frac{p}{2}}\bigg)^\frac{2}{p} \\
    & + \sum_{i,j, k = 1}^d \bigg[\Big\|\int_\cdot^T e^{ -\li s}\partial_k H^i_l(Y_s)\D^j_\cdot Y^k_sdW^l_s\Big\|_\Hil\bigg]_p \\
    & + \sum_{i,j, k = 1}^d \eps^2\bigg[\Big\| \int_\cdot^T e^{ -\li s}\partial_k G^i(Y_s)\D^j_\cdot Y^k_sds\Big\|_\Hil\bigg]_p.
\end{align*}
The terms in the first sum of the above display appeared in (\ref{eq:pre_0th_derivative_estimate}), and thus are $\lesssim \eps^{\delta_0}\big(1+\pp(\eps^{1-\upsilon}|y|)\big)^2$.
For the next two sums, we first invoke properties \eqref{eq:pBracket_property_1}, \eqref{eq:pBracket_property_2} and \eqref{eq:pBracket_property_3}, and then apply \eqref{eq:1st_Y_derivative_estimate} and  (\ref{eq:choice_of_theta}) to get 
\begin{align*}
    &\bigg[\Big\|\int_\cdot^T e^{ -\li s}\partial_k H^i_l(Y_s)\D^j_\cdot Y^k_sdW^l_s\Big\|_\Hil\bigg]_p \leq \int_0^T \bigg[\int_r^T e^{ -\li s}\partial_k H^i_l(Y_s)\D^j_r Y^k_sdW^l_s\bigg]_pdr\\
    &\lesssim \int_0^T \int_r^T \sum_{l}\big[e^{ -\li s}\partial_k H^i_l(Y_s)\D^j_r Y^k_s\big]_p ds dr \lesssim \int_0^T \int_r^T e^{-2\li s}\pbracket{\D^j_r Y^k_s} ds\, dr \lesssim \eps T \leq \eps^\frac{1}{2}
\end{align*}
and similarly,
\begin{align*}
        \eps^2\bigg[\Big\| \int_\cdot^T e^{ -\li s}\partial_k G^i(Y_s)\D^j_\cdot Y^k_sds\Big\|_\Hil\bigg]_p  \lesssim \eps^2 T\int_0^T \int_r^T e^{-2\li s}\pbracket{\D^j_r Y^k_s} ds\, dr \leq \eps^2.
\end{align*}

Therefore, we conclude that for some $\delta_1>0$,
\begin{align}\label{eq:1st_derivative_norm_est}
    \pbracket{\|\D N_T\|_\Hil} \lesssim \eps^{\delta_1}\big(1+\pp(\eps^{1-\upsilon}|y|)\big)^2, \quad p\geq 2, \quad \eps \in(0,\eps_0].
\end{align}

\subsubsection{$2$nd order derivatives}
Since $2$nd order derivatives of $Z_t$ vanish as can be seen in (\ref{eq:BasicDerivativeCalc}), we have 
\begin{equation}
\label{eq:DN=DU}
\D^{(2)} N_t = \D^{(2)} U_t,
\end{equation}
 where the superscript indicates the order of differentiation. So we only need to study the latter. 

Let us rewrite~\eqref{eq:Y_after_Duhamel} as 
\begin{align*}U^j_t&=\int_0^t e^{-\lj s}F^j_l(Y_s)dW^l_s+\eps\int^t_0e^{-\lj s}G^j(Y_s)ds\\
\notag  
   &=\int_0^t e^{-\lj s}F^j_l\big(\eps e^{\lambda s}(y+U_s)\big)dW^l_s+\eps\int^t_0e^{-\lj s}G^j\big(\eps e^{\lambda s }(y+U_s)\big)ds   
\end{align*}
and apply formula (2.54) in~\cite[Section 2.2]{nualart} to this equation
in place of equation~(2.37) therein.  For $r_1, r_2\leq t \leq T$, we obtain
\begin{align*}
     &\D^{j_1,j_2}_{r_1,r_2}U^i_t \quad =\quad  e^{-\li r_1}\partial_k F^i_{j_1}(Y_{r_1})\D^{j_2}_{r_2}Y^k_{r_1}+e^{-\li r_2}\partial_k F^i_{j_2}(Y_{r_2})\D^{j_1}_{r_1}Y^k_{r_2} \\
    & + \int_{r_1\vee r_2}^t e^{-\li s }\big(\partial^2_{k_1,k_2} F^i_l(Y_s)\big)(\D^{j_1}_{r_1}Y^{k_1}_s)(\D^{j_2}_{r_2}Y^{k_2}_s)dW^l_s + \eps\int_{r_1\vee r_2}^t e^{(\lk-\li) s }\partial_{k} F^i_l(Y_s)\D^{j_1,j_2}_{r_1,r_2}U^k_s  dW^l_s\\
    & +  \eps\int_{r_1\vee r_2}^t e^{-\li s }\big(\partial^2_{k_1,k_2} G^i(Y_s)\big)(\D^{j_1}_{r_1}Y^{k_1}_s)(\D^{j_2}_{r_2}Y^{k_2}_s)ds + \eps^2\int_{r_1\vee r_2}^t e^{(\lk-\li) s }\partial_{k} G^i(Y_s)\D^{j_1,j_2}_{r_1,r_2}U^k_s  ds.
\end{align*} 
Here we choose to express some terms only in terms of the process $Y$ while some terms are expressed in terms of both $U$ and $Y$ (we recall that by~\eqref{eq:Y_after_Duhamel} $Y_t^j=\eps e^{\lambda_j t}(y^j+ U^j_t)$).

This, along with (\ref{eq:pBracket_property_1}), \eqref{eq:pBracket_property_2}, the Cauchy--Schwarz inequality  and the boundedness of derivatives of $F$ and $G$, implies that
\begin{align}\label{eq:treatment_of_2nd_derivative}
    \begin{split}
        \pbracket{\D^{j_1,j_2}_{r_1,r_2}U^i_t} &\lesssim e^{-2\li r_1}\pbracket{\D^{j_2}_{r_2}Y^k_{r_1}} + e^{-2\li r_2}\pbracket{\D^{j_1}_{r_1}Y^k_{r_2}}  \\
    &+ (1+\eps^2T)\sum_{k_1,k_2 = 1}^d\int_{r_1\vee r_2}^t e^{-2\li s } \pbracketX{\D^{j_1}_{r_1}Y^{k_1}_s}{2p} \pbracketX{\D^{j_2}_{r_2}Y^{k_2}_s}{2p} ds\\
    &+(\eps^2+\eps^4T) \sum_{k=1}^d  \int_{r_1\vee r_2}^t e^{2(\lk-\li) s }\pbracket{\D^{j_1,j_2}_{r_1,r_2}U^k_s}  ds.
    \end{split}
\end{align}
Let us temporarily fix $j_1,j_2,r_1,r_2$ and set $a^i(t) = \pbracket{\D^{j_1,j_2}_{r_1,r_2}U^i_t} $ and $r=r_1\wedge r_2$. Then, using \eqref{eq:choice_of_theta} and (\ref{eq:1st_Y_derivative_estimate}) for $p$ and $2p$, from the above display we obtain
\begin{align*}
    a^i(t) &\lesssim  e^{-2\li r_1}\eps +  e^{-2\li r_2} \eps + \int_{r_1\vee r_2}^t e^{-2\li s } \eps^2 ds + \eps^2 \sum_{k=1}^d  \int_{r_1\vee r_2}^t e^{2(\lk-\li) s }a^k(s)  ds \\
    & \lesssim \eps e^{-2\li r} + \eps^2 \sum_{k=1}^d  \int_{r}^t e^{2(\lk-\li) s }a^k(s)  ds,
\end{align*}
which by Lemma \ref{lemma:solve_inequality_system} implies
\begin{align}\label{eq:2nd_derivative_U}
     \pbracket{\D^{j_1,j_2}_{r_1,r_2}U^i_t} &= a^i(t)\le C \eps e^{-2\lam_d r} , \quad r_1,r_2\le t \leq T, \quad p \geq 2, \quad \eps \in(0,\eps_0].
\end{align}
This, along with \eqref{eq:Y_after_Duhamel} and (\ref{eq:choice_of_theta}) implies the following estimate  which will used  
 later:
\begin{align} \label{eq:2nd_Y_derivative_estimate}
     \pbracket{\D^{j_1,j_2}_{r_1,r_2}Y^i_t} =\eps^2 e^{2 \li t}\pbracket{\D^{j_1,j_2}_{r_1,r_2}U^2_i}  \lesssim \eps^3 e^{2\lam_1 T}  \lesssim \eps^2, \quad r_1,r_2\le t \leq T, \quad p \geq 2, \quad \eps \in(0,\eps_0]. 
\end{align}
Lastly we obtain, by \eqref{eq:pBracket_property_3}, \eqref{eq:DN=DU} and  (\ref{eq:2nd_derivative_U}),
\begin{align}\label{eq:2nd_derivative_norm_est}
    \begin{split}
        \pbracket{\|\D^{(2)} N_T \|_{\Hil^{\otimes 2} }} & = \pbracket{\|\D^{(2)} U_T \|_{\Hil^{\otimes  2}}} \lesssim  \sum_{i,j_1,j_2=1}^d\int_{[0,T]^2}\pbracket{\D^{j_1,j_2}_{r_1,r_2}U^i_T}dr_1dr_2\\
    &\lesssim \eps \sum_{i=1}^d\int_{[0,T]^2} e^{-2\lam_d r_1\wedge r_2} dr_1dr_2 \lesssim \eps T \leq \eps^\frac{1}{2}, \quad p \geq 2.
    \end{split}
\end{align}

\subsubsection{$3$rd order derivatives} 
Similarly to the above argument for second order derivatives, we apply
~(2.54) from~\cite[Section~2.2]{nualart} to obtain that
for $r_1,r_2,r_3\leq t\leq T$,
\begin{align*}
    &\D^{j_1,j_2,j_3}_{r_1,r_2,r_3}U^i_t \\
    &= \frac{1}{2}\sum_{\{n_0,n_1,n_2\} = \{1,2,3 \}}e^{-\li r_{n_0}}\bigg(\partial^2_{k_1,k_2}F^i_{j_{n_0}}(Y_{r_{n_0}})\prod_{m=1}^2\D^{j_{n_m}}_{r_{n_m}}Y^{k_m}_{r_{n_0}}  + \partial_{k}F^i_{j_{n_0}}(Y_{r_{n_0}})\D^{j_{n_1},j_{n_2}}_{r_{n_1,n_2}}Y^{k}_{r_{n_0}} \bigg) \\
    & + \int_{r_1\vee r_2\vee r_3}^te^{-\li s}\bigg(\partial^3_{k_1,k_2,k_3}F^i_l(Y_s)\prod_{m=1}^3\D^{j_{m}}_{r_{m}}Y^{k_m}_s\\
    & +  \frac{1}{2}\sum_{\{n_1,n_2,n_3\} = \{1,2,3\}}\partial^2_{k_1,k_2}F^i_l(Y_s)(\D^{j_{n_1},j_{n_2}}_{r_{n_1},r_{n_2}}Y^{k_1}_s)(\D^{j_{n_3}}_{r_{n_3}}Y^{k_2}_s) + \eps e^{\lk s} \partial_{k}F^i_l(Y_{s})\D^{j_{1},j_{2},j_{3}}_{r_{1},r_{2},r_{3}}U^{k}_s\bigg)dW^l_s\\&+\eps\Big(\text{a similar integral  with $F^i_l$ and $ dW^l_s$ replaced by $G^i $ and $ds$, respectively}\Big),
\end{align*}
where the factor of $1/2$ comes from counting certain terms twice.
Let us temporarily fix $j_1, j_2, j_3,r_1,r_2,r_3$ and set  $a^i(t)=\pbracket{\D^{j_1,j_2,j_3}_{r_1,r_2,r_3}U^i_t} $ and $r= r_1 \wedge r_2 \wedge r_3$. Then, similarly to~\eqref{eq:treatment_of_2nd_derivative}, using H\"older's inequality, the $[\cdot]_p$ properties \eqref{eq:pBracket_property_1}, \eqref{eq:pBracket_property_2}, estimates  (\ref{eq:1st_Y_derivative_estimate}) and (\ref{eq:2nd_Y_derivative_estimate}) for $p,2p, 3p$, and lastly \eqref{eq:choice_of_theta}, we obtain 
\begin{align*}
    a^i(t) & \lesssim e^{-2\li r}\Big(\eps^2 + \eps^2 \Big) + \int_r^t e^{-2\li s}\Big(\eps^3 + \eps^3 + \sum_{k=1}^d\eps^2 e^{2\lk s} a^k(s) \Big) ds\\
    & \lesssim \eps^2 e^{-2\li r} + \eps^2 \sum_{k=1}^d \int_r^t e^{2(\lk-\li) s}a^k(s)ds,
\end{align*} 
which by
Lemma \ref{lemma:solve_inequality_system} yields
\begin{align*}
     \pbracket{\D^{j_1,j_2,j_3}_{r_1,r_2,r_3}U^i_t} &= a^i(t)\le C \eps^2 e^{-2\lam_d r}, \quad r_1,r_2,r_3\le t \leq T, \quad p \geq 2.
\end{align*}
Finally, by \eqref{eq:pBracket_property_3} and  (\ref{eq:choice_of_theta}) we have, with $r=r_1\wedge r_2\wedge r_3$,
\begin{align}\label{eq:3rd_derivative_norm_est}
    \begin{split}
        &\pbracket{\|\D^{(3)} U_T \|_{\Hil^{\otimes 3}}}   \lesssim  \sum_{i,j_1,j_2,j_3=1}^d\int_{[0,T]^3}\pbracket{\D^{j_1,j_2,j_3}_{r_1,r_2,r_3}U^i_T}dr_1dr_2dr_3\\
    &\lesssim \eps^2 \sum_{i=1}^d\int_{[0,T]^3} e^{-2\lam_d r}dr_1dr_2dr_3\lesssim \eps^2 T^2 \leq \eps, \quad p\geq 2, \quad \eps \in(0,\eps_0].
    \end{split}
\end{align}
\subsubsection{Conclusion of derivative estimates}
Combining estimates \eqref{eq:0th_derivative__norm_est}, \eqref{eq:1st_derivative_norm_est}, \eqref{eq:2nd_derivative_norm_est},  and Jensen's inequality, we obtain, for each $p\geq 1$, all $\eps \in(0,\eps_0]$,
\begin{align} \label{eq:Derivative_difference_est}
    \|U_T-Z_T \|_{2,p} = \|N_T \|_{2,p}=\pbracket{N_T}^\frac{1}{2}+ \sum^2_{k=1}\pbracket{\|\D^kN_T\|_{\Hil^{\otimes k}}}^\frac{1}{2}\lesssim \eps^{\delta}\big(1+\pp(\eps^{1-\upsilon}|y|)\big)
\end{align}
for some $\delta >0 $.

By \eqref{eq:1st_Derivative_U}, \eqref{eq:2nd_derivative_U}, \eqref{eq:3rd_derivative_norm_est}, Jensen's inequality and the easy observation that all moments of $U_t$ are bounded uniformly in $t$, we have
\begin{align}\label{eq:Sobolev_norm_bound_U_T}
    \|U_T\|_{3,p}\lesssim 1,\quad p\geq 1, \quad \eps\in(0,\eps_0].
\end{align}
Lastly, a simple calculation shows that 
\begin{align}\label{eq:Sobolev_norm_bound_Z_T}
    \|Z_T\|_{3,p}\lesssim 1, \quad  p\geq 1, \quad \eps\in(0,\eps_0].
\end{align}

\subsection{Negative moments for determinants of Malliavin matrices \texorpdfstring{$\sigma_{U_T}$}{sigma\_\{U\_T\}} and~\texorpdfstring{$\sigma_{Z_T}$}{sigma\_\{Z\_T\}}} 
The goal is to show for each $p\geq 1$ there is a $C_p>0$ such that
\begin{align}\label{eq:goal_est_det_malliavin_covariance}
    \EE|\det \sigma_{U_T}|^{-p},\  \EE|\det \sigma_{Z_T}|^{-p} \leq C_p,  \quad\eps\in(0,\eps_0].
\end{align}

Using the formula of $\D Z_t$ in \eqref{eq:BasicDerivativeCalc} and that $F(0)=\sigma(0)$ is of full rank, it is easy to verify \eqref{eq:goal_est_det_malliavin_covariance} for $ \sigma_{Z_T}$ as it is deterministic. For $\sigma_{U_T}$, we first simplify the expression for $\D^j_rU^i_t$ in \eqref{eq:formula_of_1st_derivative_U}. Let
\begin{gather}\label{eq:def_A_Abar_Bbar}
    A^i_j(r)=e^{-\li r}F^i_j(Y_r), \quad \bA^i_{k,l}(s)=\eps e^{(\lk-\li)s}\partial_k F^i_l(Y_s),\quad 
    \bB^i_k(s)=\eps^2 e^{(\lk-\li)s}\partial_k G^i(Y_s).
\end{gather}
Then, we can rewrite \eqref{eq:formula_of_1st_derivative_U} as
\begin{align*}
    \D^j_rU^i_t = A^i_j(r) + \int_r^t \bA^i_{k,l}(s)\D^j_rU^k_sdW^l_s +\int_r^t \bB^i_k(s)\D^j_rU^k_s ds.
\end{align*}
By the boundedness of derivatives of $F$, $G$ and \eqref{eq:choice_of_theta}, we have, for some $C>0$, 
\begin{align}\label{eq:bound_A_bA_bB}
    |A^i_j(s)|\leq Ce^{-\li s},\quad |\bA^i_{k,l}(s)|\leq C\eps^\frac{1}{2},\quad |\bB^i_k(s)|\leq C\eps^\frac{3}{2}, \quad s\leq T, \quad \eps\in(0,\eps_0].
\end{align}
Let us introduce two $d\times d$-matrix-valued processes, where $\delta^i_j$ is the Kronecker delta,
\begin{align}\label{eq:def_YY_ZZ}
    \begin{split}
        \YY^i_j(t)&=\delta^i_j + \int_0^t\Big(\bA^i_{k,l}(s)\YY^k_j(s)dW^l_s+\bB^i_k(s)\YY^k_j(s)ds\Big),\\
        \ZZ^i_j(t)&= \delta^i_j -\int_0^t \Big(\bA^k_{j,l}\ZZ^i_k(s)dW^l_s+\big(\bB^k_j(s)-\sum_{l=1}^d\bA^k_{m,l}(s)\bA^m_{j,l}(s)\big)\ZZ^i_k(s)ds\Big). 
    \end{split}
\end{align}
These two processes correspond to (2.57) and (2.58) in \cite[Section 2.3.1]{nualart}. The computations below (2.58) there show $\ZZ(t)\YY(t)=\YY(t)\ZZ(t)=I$ the identity matrix. In addition, (2.60) and (2.61) from \cite[Section 2.3.1]{nualart} state that $\sigma_{U_t}$ satisfies
\begin{align}\label{eq:sigma_U_T_formula}
    \begin{split}
        \sigma_{U_t}=\YY(t)\CC_t\YY(t)^\intercal
    \end{split}
\end{align}
where $\intercal$ denotes the matrix transpose operation and 
\begin{align}\label{eq:def_of_C_t}
    \CC^{ij}_t=\sum_{l=1}^d \int_0^t \ZZ^i_k(s)A^k_l(s)\ZZ^j_m(s)A^m_l(s)ds.
\end{align}

Then, observe that, by \eqref{eq:sigma_U_T_formula} and H\"older's inequality, 
\begin{align}\label{eq:holder_ineq_det_malliavin_covariance}
    \begin{split}
        \EE|\det \sigma_{U_T}|^{-p} \leq \big(\EE|\det \CC_T|^{-2p}\big)^\frac{1}{2}\big(\EE|\det \ZZ(T)|^{4p}\big)^\frac{1}{2}, \quad p\geq 1.
    \end{split}
\end{align}
Therefore, to prove boundedness of $\EE|\det \sigma_{U_T}|^{-p}$, it suffices to prove that it holds for $\EE|\det \CC_T|^{-2p}$ and $\EE|\det \ZZ(T)|^{4p}$. We first bound the latter. 

Although it is more than what we need here, we shall find a bound on moments of $\tZZ(T)$ with $\tZZ^i_j(t)= \sup_{0\leq r\leq t}|\ZZ^i_j(t)| $, which will be used later.
By \eqref{eq:bound_A_bA_bB}, we have
\begin{align*}
    \tZZ^i_j(T) \lesssim 1 + \sup_{0\leq r\leq T}\Big|\int_0^r \bA^k_{j,l}\ZZ^i_k(s)dW^l_s\Big|+\int_0^T \eps \sum_{k=1}^d\tZZ^i_k(s)ds.
\end{align*}
Then, using BDG inequality, the $[\cdot]_p$ properties \eqref{eq:pBracket_property_1} and \eqref{eq:pBracket_property_2},  \eqref{eq:choice_of_theta} and \eqref{eq:bound_A_bA_bB}, we obtain,
for all $p\geq 2$ and $\eps\in(0,\eps_0]$,
\begin{align*}
        \pbracket{\tZZ^i_j(T)}&\lesssim 1+\sum_{k=1}^d\int_0^T\big(\eps + \eps^2T\big)\pbracket{\tZZ^i_k(s)}ds \lesssim 1+\eps \sum_{k=1}^d\int_0^T\pbracket{\tZZ^i_k(s)}ds.
\end{align*}
Summing up the above in $j$ and using Gronwall's inequality, we get, for some $c>0$,
\begin{align}\label{eq:p-moment_est_ZZ_and_tZZ}
    \pbracket{\tZZ^i_j(T)} \leq \sum_{k=1}^d \pbracket{\tZZ^i_k(T)} \lesssim e^{c\eps T}\lesssim  1,\quad p\geq 2, \quad \eps \in(0,\eps_0].
\end{align}
Using this and the expression of the matrix determinant as a polynomial of the entries, we apply H\"older's inequality to conclude that for each $p\geq 1$ there is $C_p>0$ such that $\big(\EE|\det \ZZ(T)|^{4p}\big)^\frac{1}{2}\leq C_p,$ $\eps \in(0,\eps_0]$.

\medskip

To bound $\EE|\det \CC_T|^{-2p}$ for all $p\geq1$, it suffices to show that, for each $p\geq 1$, there is $C_p>0$ such that $\Prob{\nu \leq \zeta}\leq C_p \zeta^p$, where $\nu$ is the smallest eigenvalue of $\CC_T$. Note that $\nu\geq 0$, because $\CC_T$  is positive semi-definite, which can be derived from \eqref{eq:sigma_U_T_formula} since $\sigma_{U_T}$ is positive semi-definite and $\YY(T)$ is invertible. We need the following lemma which will be proved in Section \ref{section:auxiliary_lemmas}.
\begin{Lemma}\label{lemma:smallest_eigenvalue}
Let $\mathcal{A}$ be a symmetric positive semi-definite
random $d\times d$ matrix. Let~$\nu$ be its smallest eigenvalue. Then for each $p\geq 1$, there is a constant $C_{p,d}>0$ such that
\begin{align}
\label{eq:smallest_eig}
    \Prob{\nu \leq \zeta}\leq C_{p,d} \bigg(\sup_{|v|=1} \E{|\langle v, \mathcal{A}v\rangle|^{-(p+2d)} }+ \EBig{\Big|\sum_{i,j =1}^d |\mathcal{A}^{ij}|^2\Big|^\frac{p}{2}}\bigg)\zeta^p,\quad \zeta\geq 0.
\end{align}
\end{Lemma}

For each $p> 1$, by \eqref{eq:bound_A_bA_bB},  \eqref{eq:p-moment_est_ZZ_and_tZZ} and H\"older's inequality, we have
\begin{align*}
    \big(\EE|\CC^{ij}_T|^p\big)^\frac{1}{p}&\leq  \int_0^T \big(\EE|\ZZ^i_k(s)A^k_l(s)\ZZ^j_m(s)A^m_l(s)|^p\big)^\frac{1}{p}ds\\
    &\lesssim \sum_{k,m=1}^d \int_0^T e^{-\lk s}e^{-\lambda_m s}ds \lesssim 1, \quad \eps \in (0,\eps_0].
\end{align*}
Hence, for each $p\geq 1$, there is $c_p>0$ such that $\EBig{\big|\sum_{i,j =1}^d |\CC_T^{ij}|^2\big|^\frac{p}{2}}\leq c_p$, $\eps \in (0,\eps_0]$. Therefore, if we can show that for each $p\geq 1$ there is $C_p$ such that
\begin{align}\label{eq:key_inverse_p_moment_est_for_malliavin_covariance}
    \sup_{|v|=1} \E{|\langle v, \CC_Tv\rangle|^{-p} }\leq C_p,  \quad \eps\in(0,\eps_0],
\end{align}
then Lemma \ref{lemma:smallest_eigenvalue} and the discussion above imply that $\EE|\det \CC_T|^{-p}$ is bounded for each $p\geq 1$, when $\eps$ is small (in comparison with \cite[Lemma 2.3.1]{nualart}, we need a bound that is uniform in $\eps\in(0,\eps_0]$). Consequently, this and  \eqref{eq:holder_ineq_det_malliavin_covariance} will imply the desired result~\eqref{eq:goal_est_det_malliavin_covariance}. Therefore, it remains to show \eqref{eq:key_inverse_p_moment_est_for_malliavin_covariance}.

\begin{proof}[Proof of \eqref{eq:key_inverse_p_moment_est_for_malliavin_covariance}]
Let us fix an arbitrary $v\in \Sph^{d-1}$, the $(d-1)$-sphere. By the definition of $\CC_t$ given in  \eqref{eq:def_of_C_t},
\begin{align*}
    \langle v, \CC_T v\rangle = \int_0^T |v^\intercal \ZZ(s)A(s)|^2 ds.
\end{align*}
Recall $A^i_j(s)$ given in \eqref{eq:def_A_Abar_Bbar}. By \eqref{eq:modified_F}, we have, in the sense of positive semi-definite matrices, $F(Y_s)F(Y_s)^\intercal\geq c_0 I$. Therefore, we get
\begin{align}\label{eq:<v,C_Tv>_geq_c_0_integral}
    \langle v, \CC_T v\rangle \geq  c_0 \int_0^T \Big|\Big(\sum_{i=1}^dv_i \ZZ^i_j(s)e^{-\lj s}\Big)_{j=1}^d\Big|^2 ds, \quad \eps\in(0,\eps_0].
\end{align}
Let us  define
\begin{align}\label{eq:def_R^j_t}
    \begin{split}
        R^j_t&= \sqrt{c_0}\sum_{i=1}^d v_i \ZZ^i_j(t)e^{-\lj t} = r^j_0+M^j_t+A^j_t \\
        &= r^j_0 + \int_0^t u^j_l(s)dW^l_s+\int_0^t a^j(s)ds, \quad j=1,2,\dots,d,
    \end{split}
\end{align}
and additionally $N^j_t=\int_0^t R^j(s)u^j_l(s)dW^l_s, \quad j=1,2,\dots, d$, 
where, by the It\^o formula and the expression for $\ZZ(t)$ given in \eqref{eq:def_YY_ZZ},
\begin{align}
    r_0^j= \sqrt{c_0}v_j, \quad\quad\quad &u^j_l(s)=-\sqrt{c_0}\sum_{i,k=1}^dv_ie^{-\lj s}\bA^k_{j,l}(s)\ZZ^i_k(s),\label{eq:def_r_0}\\
    a^j(s)=-\Big(\sqrt{c_0}\sum_{i=1}^d v_i \lj e^{-\lj s}\ZZ^i_j(s)\Big)-&\Big(\sqrt{c_0}\sum_{i,k,m,l} v_ie^{-\lj s}\big(\bB^k_j(s)-\bA^k_{m,l}(s)\bA^m_{j,l}(s)\big)\ZZ^i_k(s)\Big).\nonumber
\end{align}
Then, \eqref{eq:<v,C_Tv>_geq_c_0_integral} and \eqref{eq:def_R^j_t} imply
\begin{align*}
    \Prob{\langle v, \CC_T v\rangle\leq \zeta} \leq \ProbBig{\int_0^T|R_s|^2ds \leq \zeta}, \quad \eps\in(0,\eps_0].
\end{align*}
Recall that $T=T(\eps)\geq 1$ is assumed. Since $v\in \Sph^{d-1}$ is arbitrary, Lemma \ref{lemma:near_zero_probability_est} 
that we state and prove below implies \eqref{eq:key_inverse_p_moment_est_for_malliavin_covariance}.
\end{proof}
\begin{Lemma}\label{lemma:near_zero_probability_est}
Let $\eps_0$ be given in \eqref{eq:choice_of_theta}, and $R_s$ be given in \eqref{eq:def_R^j_t} which depends on the choice of $v\in \Sph^{d-1}$. For each $p\geq 1$, there is $C_p>0$ independent of $v$ such that
\begin{align}\label{eq:desired_result_lemma_near_zero_probability_est}
    \ProbBig{\int_0^T|R_s|^2ds \leq \zeta}\leq C_p \zeta^{\frac{1}{8}p}, \quad \eps\in(0,\eps_0].
\end{align}
\end{Lemma}

This lemma is a variation of \cite[Lemma 2.3.2]{nualart}.
\begin{proof}[Proof of Lemma \ref{lemma:near_zero_probability_est}]
By \eqref{eq:bound_A_bA_bB}, \eqref{eq:p-moment_est_ZZ_and_tZZ} and \eqref{eq:def_r_0}, there is  $c_p>0$ independent of $v$ such that
\begin{align*}
    \EE\sup_{0\leq s\leq T}|u(s)|^p\leq c_p ;\quad \EE\Big(\int_0^T|a(s)|^2ds\Big)^p\leq c_p, \quad \eps\in(0,\eps_0].
\end{align*}
This and Markov's inequality imply that for some $c_p>0$ independent of $v\in \Sph^{d-1}$,
\begin{align}\label{eq:tail_est_sup_theta_s}
     \ProbBig{\sup_{0\leq s\leq T}\Big(|u(s)|+\int_0^s|a(r)|^2dr\Big) > \zeta^{-\frac{1}{8}}}\leq c_p\zeta^{\frac{1}{8}p},\quad \zeta>0,\  \eps\in(0,\eps_0].
\end{align}

Recalling the definitions of $M_t$ and $N_t$ in \eqref{eq:def_R^j_t}, we define, for each $\zeta>0$ and each $\eps \in(0,\eps_0]$,
\begin{align*}
    B_{0}^{\zeta,\eps}&=\Big\{\int_0^T|R_s|^2ds \leq \zeta,\ \sup_{0\leq s\leq T}\Big(|u(s)|+\int_0^s|a(r)|^2dr\Big)\leq \zeta^{-\frac{1}{8}}\Big\},\\
    B^{\zeta, \eps}_{1,j}&=\{\langle M^j\rangle_T \leq 2\zeta^\frac{1}{4}, \sup_{0\leq t\leq T}|M^j_t|\geq d^{-1}\zeta^\frac{1}{16}\},\\
    B^{\zeta, \eps}_{2,j}&=\{\langle N^j\rangle_T \leq \zeta^\frac{3}{4}, \sup_{0\leq t\leq T}|N^j_t|\geq \zeta^\frac{1}{4}\},
\end{align*}
where the dependence on $\eps$ comes from $T=T(\eps)$, $R_s$, $u(s)$, $a(s)$, $M_s$, and $N_s$.
The exponential martingale inequality implies that, for some $c_p>0$ in dependent of $v$,
\begin{align*}
    \Prob{(\cup_{j=1}^d B^{\zeta,\eps}_{1,j})\cup (\cup_{j=1}^d B^{\zeta,\eps}_{2,j})} \leq 2d \exp(-\tfrac{\zeta^{-\frac{1}{8}}}{4d^2})+ 2d \exp(-\tfrac{\zeta^{-\frac{1}{4}}}{2})\leq c_p \zeta^{\frac{1}{8}p}, \quad \zeta>0,\ \eps\in(0,\eps_0].
\end{align*}

Observe that by this and \eqref{eq:tail_est_sup_theta_s}, we can attain the desired result \eqref{eq:desired_result_lemma_near_zero_probability_est} if we can show there is a $\zeta_0>0$ such that
\begin{align}\label{eq:key_set_inclusion}
    B^{\zeta,\eps}_{0}\subset  (\cup_{j=1}^d B^{\zeta,\eps}_{1,j})\cup (\cup_{j=1}^d B^{\zeta,\eps}_{2,j}), \quad \zeta \in (0, \zeta_0),\ \eps\in(0,\eps_0].
\end{align}

Hence, it remains to show \eqref{eq:key_set_inclusion}. Choose a  $\zeta_0$ to satisfy, with $c_0, r_0$ given in \eqref{eq:def_r_0},
\begin{align}\label{eq:zeta_0_condition}
    2d(\zeta_0^\frac{1}{4}+\zeta_0^\frac{7}{16})\leq c_0 = |r_0|^2, \quad 4\zeta_0^\frac{1}{16}\leq \sqrt{c_0}, \quad \text{ and }\quad \zeta_0^\frac{1}{3}\leq \frac{1}{2}.
\end{align}

We show \eqref{eq:key_set_inclusion} with this chosen $\zeta_0$.
Argue by contradiction. Suppose \eqref{eq:key_set_inclusion} is false. Then, for some $\zeta \in (0,\zeta_0)$ and some $\eps\in(0,\eps_0]$, there is 
\begin{align}\label{eq:omega_condition}
    \omega \in  B^{\zeta,\eps}_{0}-  \Big((\cup_{j=1}^d B^{\zeta,\eps}_{1,j})\cup (\cup_{j=1}^d B^{\zeta,\eps}_{2,j})\Big).
\end{align}
From now on, fix this pair of $\zeta$ and $\eps$, and evaluate all random variables at this $\omega$. 

By $\omega \in  B^{\zeta,\eps}_{0}$ due to \eqref{eq:omega_condition}, we clearly have
\begin{align*}
    \qd{N^j}_T \leq \int_0^T|R^j_s u^j(s)|^2 ds \leq \Big(\sup_{0\leq s\leq T}|u(s)|^2\Big)\int_0^T|R^j_s|^2ds \leq \zeta^{-\frac{2}{8}+1}=\zeta^\frac{3}{4}.
\end{align*}
Then, since $\omega \not\in  B^{\zeta,\eps}_{2,j}$, $j=1,2,\dots,d$, due to \eqref{eq:omega_condition}, we deduce
\begin{align*}
    \sup_{0\leq t\leq T}\Big|\int_0^tR^j_su^j_l(s)dW^l_s\Big|=\sup_{0\leq t\leq T}|N^j_t|<\zeta^\frac{1}{4}, \quad j=1,2,\dots,d.
\end{align*}
By $\omega \in  B^{\zeta,\eps}_{0}$ due to \eqref{eq:omega_condition} and the Cauchy--Schwarz inequality, we have
\begin{align*}
    \sup_{0\leq t\leq T}\Big|\int_0^tR^j_sa^j(s)ds\Big|\leq \Big(\int_0^T|R^j_s|^2ds\Big)^\frac{1}{2}\Big(\int_0^T|a^j(s)|^2ds\Big)^\frac{1}{2}\leq \zeta^{\frac{1}{2}-\frac{1}{16}}=\zeta^\frac{7}{16}.
\end{align*}
It\^o formula applied to \eqref{eq:def_R^j_t} gives $|R_t|^2 = |r_0|^2 +\sum_{j=1}^d 2(\int_0^t R_s^j u^j_ldW^l_s + \int_0^tR^j_sa^j(s)ds)+\sum_{j=1}^d\qd{M^j}_t$. The above two displays, \eqref{eq:zeta_0_condition}, and $\omega \in  B^{\zeta,\eps}_{0}$ due to \eqref{eq:omega_condition} imply
\begin{align*}
    \int_0^T\sum_{j=1}^d\qd{M^j}_tdt &= \int_0^T|R_s|^2 dt-T|r_0|^2  - \int_0^T \Big(\sum_{j=1}^d 2\int_0^t R_s^j dR^j_s\Big) dt \\
    &\leq \zeta - T|r_0|^2 + 2dT(\zeta^\frac{1}{4}+\zeta^\frac{7}{16}) \leq \zeta.
\end{align*}
Because $t\mapsto \sum_{j=1}^d\qd{M^j}_t$ is nondecreasing, the above display indicates
\begin{align*}
    \gamma \sum_{j=1}^d\qd{M^j}_{T-\gamma}\leq\zeta, \quad \gamma \leq 1 \leq T.
\end{align*}
Since $\omega \in  B^{\zeta,\eps}_{0}$ implies $\sup_{0\leq s\leq T}|u(s)|\leq \zeta^{-\frac{1}{8}}$, by the definition of $M_t$ in \eqref{eq:def_R^j_t}, we get
\begin{align*}
    \sum_{j=1}^d(\qd{M^j}_T-\qd{M^j}_{T-\gamma})\leq \gamma \zeta^{-\frac{1}{4}}.
\end{align*}
The above two displays yield $\sum_{j=1}^d\qd{M^j}_T\leq \gamma^{-1}\zeta + \gamma \zeta^{-\frac{1}{4}}$. By \eqref{eq:zeta_0_condition}, we can set $\gamma = \zeta^\frac{1}{2} <\zeta_0^\frac{1}{2}<1$ to obtain
\begin{align*}
    \qd{M^j}_T \leq \sum_{j=1}^d\qd{M^j}_T \leq \zeta^{-\frac{1}{2}+1}+\zeta^{\frac{1}{2}-\frac{1}{4}}\leq 2\zeta^\frac{1}{4}.
\end{align*}
Since $\omega \not \in B^{\zeta,\eps}_{1,j}$, $j=1,2,\dots,d$, due to \eqref{eq:omega_condition}, we have
\begin{align}\label{eq:sup_|M_t|_est}
    \sup_{0\leq t\leq T}|M_t| < \sum_{j=1}^d \sup_{0\leq t\leq T}|M^j_t| < dd^{-1}\zeta^\frac{1}{16}=\zeta^\frac{1}{16}.
\end{align}
On the other hand, Markov's inequality and $\omega \in B^{\zeta,\eps}_{0}$ imply,
\begin{align*}
    \mathbf{m}\{t\in[0,T]:|R_t|\geq \zeta^\frac{1}{3}\}\leq \frac{1}{\zeta^\frac{2}{3}}\int_0^T|R_t|^2dt \leq \zeta^\frac{1}{3},
\end{align*}
where $\mathbf{m}$ is the Lebesgue measure on the real line. By \eqref{eq:sup_|M_t|_est} and \eqref{eq:def_R^j_t}, we thus have
\begin{align*}
    \mathbf{m}\{t\in[0,T]:|r_0+A_t|\geq \zeta^\frac{1}{3}+\zeta^\frac{1}{16}\}\leq \zeta^\frac{1}{3}.
\end{align*}
Note that $\zeta^\frac{1}{3}<\zeta_0^\frac{1}{3}\leq \frac{1}{2} \leq \frac{1}{2}T$ due to \eqref{eq:zeta_0_condition} and $T\geq 1$. Hence, for each $t\in[0,T]$, there is $t'\in[0,T]$ satisfying $|t-t'|\leq 2\zeta^\frac{1}{3}$ and $|r_0+A_{t'}|<\zeta^\frac{1}{3}+\zeta^\frac{1}{16}$. Therefore, for each $t\in[0,T]$, it holds that, by the definition of $A_t$ in \eqref{eq:def_R^j_t},
\begin{align*}
    |r_0+A_t|&\leq |r_0+A_{t'}|+\Big|\int_{t'}^ta(s)ds\Big| < \zeta^\frac{1}{3}+\zeta^\frac{1}{16} + \Big|\int_{t'}^t|a(s)|^2ds\Big|^\frac{1}{2}|t-t'|^\frac{1}{2}\\
    &\leq \zeta^\frac{1}{3}+\zeta^\frac{1}{16}+ \sqrt{2}\zeta^{-\frac{1}{16}+\frac{1}{6}}\leq 4\zeta^\frac{1}{16}.
\end{align*}
Set $t=0$ to obtain $|r_0|<  4\zeta^\frac{1}{16}$. However, $\sqrt{c_0}=|r_0|$, due to \eqref{eq:def_r_0}, and \eqref{eq:zeta_0_condition} imply that
\begin{align*}
    \sqrt{c_0}=|r_0|<  4\zeta^\frac{1}{16} <4\zeta_0^\frac{1}{16}\leq \sqrt{c_0}.
\end{align*}
By contradiction, $\eqref{eq:key_set_inclusion}$ holds for $\zeta_0$ satisfying \eqref{eq:zeta_0_condition}.
\end{proof}

\subsection{Proof of Lemma \ref{Lemma:density_est}} \label{subsection:proof_of_density_est_lemma}
\subsubsection{Part (1)}
We will apply Theorem \ref{thm:DensityDifference} to $U_T$ and $Z_T$. First note that, since $U_t, Z_t$ are solutions of SDEs, by \cite[Theorem 2.2.2]{nualart}, we know they belong 
to~$\mathbb{D}^{3,\infty}$, see Remark~\ref{rem:smooth-coef}. 
Since $U_T=M_T+\eps V_T$, using boundedness of $V_T$ and applying exponential martingale inequality to $M_T$ and $Z_T$, after a simple computation, we have that there are constants $C,\ c>0$ such that 
\begin{align}\label{eq:tail_bound_U_T_and_Z_T}
    \Prob{|U_T-x|<2}, \quad \Prob{|Z_T-x|<2}\leq C e^{-c|x|^2}.
\end{align}

Theorem \ref{thm:DensityDifference}, \eqref{eq:Derivative_difference_est},  \eqref{eq:Sobolev_norm_bound_U_T}, \eqref{eq:Sobolev_norm_bound_Z_T},
\eqref{eq:goal_est_det_malliavin_covariance} and \eqref{eq:tail_bound_U_T_and_Z_T} give rise to, for some $C', c'>0$,
\begin{align*}
    & |\rho^y_{U_T}(x)-\rho_{Z_T}(x)|\\ &\leq C\|U_T - Z_T\|_{2, \gamma,T}\Big(\big(1\vee\E{|\det \sigma_{U_T}|^{-\gamma}}\big)\big(1+\|U_T\|_{3, \gamma,T}\big)\Big)^a\\
    &\cdot \Big(\big(1\vee\E{|\det \sigma_{Z_T}|^{-\gamma}}\big)\big(1+\|Z_T\|_{3, \gamma,T}\big)\Big)^a\cdot \big(\Prob{|U_T-x|<2}+\Prob{|Z_T-x|<2}\big)^b\\
    &\leq C' \eps^{\delta}\big(1+\pp(\eps^{1-\upsilon}|y|)\big) e^{-c'|x|^2}.
\end{align*}

\smallskip

\subsubsection{Part (2)}  We estimate the difference $|\rho_{Z_T}(x)-\rho_{Z_\infty}(x)|$. The covariance matrix of $Z_T$ is given by
\begin{align*}
    \Ceps^{jk}= \E{Z^j_TZ^k_T}=\sum_{l=1}^d\sigma^j_l(0)\sigma^k_l(0)\frac{1-e^{-(\lj+\lk)T}}{\lj + \lk}.
\end{align*}
By $T\geq \theta'_0\log\eps^{-1}$, we have $\lim_{\eps\to0}\Ceps^{jk}=\Czero^{jk}$. Therefore, there is a constant $c>0$ such that
\begin{align}\label{eq:gaussian_bound}
    e^{-\frac{1}{2}x^\intercal\Ceps^{-1}x},\quad e^{-\frac{1}{2}x^\intercal\Czero^{-1}x} \leq e^{-c|x|^2}.
\end{align}
We can write
\begin{align*}
    |\rho_{Z_T}(x)-\rho_{Z_\infty}(x)|\leq \Big|\rho_{Z_T}-\sqrt{\tfrac{\det\Ceps}{\det \Czero}}\rho_{Z_T}\Big|+\Big|\sqrt{\tfrac{\det\Ceps}{\det \Czero}}\rho_{Z_T}-\rho_{Z_\infty}\Big|.
\end{align*}
Since $\sqrt{\frac{\det\Ceps}{\det \Czero}}$ can be viewed as the square root of a polynomial of $e^{-T}$ with positive fractional powers, one can see that $|1-\sqrt{\frac{\det\Ceps}{\det \Czero}}|\leq C_1 (e^{-T})^{q_1}$ for some $C_1, q_1>0$. Therefore, using the hypothesis $\theta\log\eps^{-1}\leq T$ and \eqref{eq:gaussian_bound}, we obtain
\begin{align*}
    \Big|\rho_{Z_T}-\sqrt{\tfrac{\det\Ceps}{\det \Czero}}\rho_{Z_T}\Big|\leq C_1\eps^{q_1\theta'_0}e^{-c|x|^2}.
\end{align*}

For any matrix, we use $|\cdot|$ to denote its Frobenius norm. Then observe that, for some $q_2>0$, we have, for some $C_2,q_2>0$, 
\begin{align*}
    |\Ceps^{-1} - \Czero^{-1}|\leq |\Czero^{-1}||\Czero - \Ceps||\Ceps^{-1}|\leq C_2(e^{-T})^{q_2} \leq C_2\eps^{q_2\theta'_0}.
\end{align*}

As $\Ceps$ and $ \Czero$ are positive definite, so are their inverses. Then by \eqref{eq:gaussian_bound}, we can get
\begin{align*}
    \begin{split}
        \big|e^{-\frac{1}{2}x^\intercal\Ceps^{-1}x }-e^{-\frac{1}{2}x^\intercal\Czero^{-1}x}\big| &\leq \big(e^{-\frac{1}{2}x^\intercal\Ceps^{-1}x}\vee e^{-\frac{1}{2}x^\intercal\Czero^{-1}x}\big)\big|e^{-\frac{1}{2}|x^\intercal(\Ceps^{-1}-\Czero^{-1})x|}-1\big|\\
    &\leq \tfrac{1}{2}e^{-c|x|^2}|x|^2|\Ceps^{-1}-\Czero^{-1}|\leq C_3\eps^{q_2\theta'_0}e^{-c'|x|^2}|x|^2.
    \end{split}
\end{align*}
Therefore, we have
\begin{align*}
    \begin{split}
        \Big|\sqrt{\tfrac{\det\Ceps}{\det \Czero}}\rho_{Z_T}-\rho_{Z_\infty}\Big|\leq C_4\eps^{q_2\theta'_0}e^{-c''|x|^2}.
    \end{split}
\end{align*}
In conclusion, $ |\rho_{Z_T}(x)-\rho_{Z_\infty}(x)|\leq C'\eps^{\delta'}e^{-c'''|x|^2}$ which completes the proof of Lemma~\ref{Lemma:density_est}.

\subsection{Proofs of auxiliary lemmas}\label{section:auxiliary_lemmas}

\begin{proof}[Proof of Lemma \ref{lemma:solve_inequality_system}]
Let $b(t)= \sum_{i=1}^d a^i(t)$. Summing up the inequalities \eqref{eq:system_of_ineq} in $i$ and using $\lam_1 >\lam_2 > ... >\lam_d$, we get
\begin{align*}
    0\leq b(t)\lesssim \eps^m e^{-2\lam_d r} +\eps^2 \int_r^t e^{2\lam_1 s}b(s)ds.
\end{align*}
Now Gronwall's inequality implies that, for some constant $c$ independent of~$\eps$,
\begin{align*}
    0\leq b(t) \lesssim \eps^m e^{-2\lam_d r} e^{c\eps^2 e^{2\lam_1 T}} .
\end{align*}
Finally, we use \eqref{eq:choice_of_theta} and the fact $a^i(t)\geq 0$ to derive $a^i(t)\leq b(t)\lesssim \eps^m e^{-2\lam_d r}$, and it is clear from this computation that all the constants
involved do not depend on $r$.
\end{proof}

\bigskip

\begin{proof}[Proof of Lemma \ref{lemma:smallest_eigenvalue}]
This proof is a modification of the proof of \cite[Lemma~2.3.1]{nualart}.

Let us fix $\zeta>0$. Let $u_1, u_2, ..., u_{N_d}$ be unit vectors in $\R^d$ such that 
\begin{align}\label{eq:unit_ball_covering}
    \Sph^{d-1} \subset \cup_{k=1}^{N_d}\{x\in \R^d: |x-u_k| < \tfrac{\zeta^2}{4} \},
\end{align}
where $\Sph^{d-1}$ is the unit sphere  and  $N_d$ is chosen so that
\begin{align}\label{eq:N_d_est}
    N_d\leq C_d \zeta^{-2d}
\end{align}
for a positive constant $C_d$ only depending on the dimension $d$.
Writing $|\mathcal{A}|= (\sum_{i,j=1}^d |\mathcal{A}^{ij}|^2)^\frac{1}{2}$, we obtain
\begin{align}\label{eq:smallest_eigenvalue_est_1}
\begin{split}
     \Prob{\nu \leq \zeta} & = \Prob{\inf_{|v|=1}\langle v , \mathcal{A}v\rangle \leq \zeta} \\
    & \leq \Prob{\inf_{|v|=1}\langle v , \mathcal{A}v\rangle \leq \zeta; |\mathcal{A}|\leq \tfrac{1}{\zeta}} + \Prob{|\mathcal{A}|> \tfrac{1}{\zeta}}.
\end{split}
\end{align}
The second term can be estimated using Markov's inequality as
\begin{align}\label{eq:smallest_eigenvalue_est_2}
    \Prob{|\mathcal{A}|> \tfrac{1}{\zeta}}\leq \zeta^p \E{|\mathcal{A}|^p} = \zeta^p \EBig{\Big|\sum_{i,j=1}^d |\mathcal{A}^{ij}|^2\Big|^\frac{p}{2}}.
\end{align}
For the first term, more effort is needed. On the set 
\begin{align*}
B=\{\inf_{|v|=1}\langle v , \mathcal{A}v\rangle \leq \zeta; |\mathcal{A}|\leq \tfrac{1}{\zeta} \},
\end{align*}
suppose $\langle u_k , \mathcal{A}u_k\rangle \geq 2\zeta$ for all $k=1,\dots,N_d$. For any $v$ with $|v|=1$, by (\ref{eq:unit_ball_covering}), there is $u_k$ such that $|v-u_k|<\frac{\zeta^2}{4}$. Then observe that, on $B$,
\begin{align*}
    \langle v , \mathcal{A}v\rangle  &\geq \langle u_k , \mathcal{A}u_k\rangle  - |\langle v , \mathcal{A}v\rangle - \langle u_k , \mathcal{A}u_k\rangle | \\
    & \geq 2\zeta - \big(|\langle v , \mathcal{A}v\rangle - \langle v , \mathcal{A}u_k\rangle |+|\langle v , \mathcal{A}u_k\rangle - \langle u_k , \mathcal{A}u_k\rangle |\big) \\
    &\geq 2\zeta - 2|\mathcal{A}||v-u_k| >2\zeta - 2 \tfrac{1}{\zeta}\tfrac{\zeta^2}{4} =\tfrac{3}{2}\zeta.
\end{align*}
But on $B$, we necessarily have $\inf_{|v|=1}\langle v , \mathcal{A}v\rangle \leq \zeta$. Hence, by contradiction, we must have $B\subset\cup_{k=1}^{N_d} \{\langle u_k , \mathcal{A}u_k\rangle < 2\zeta \}$.
This fact together with \eqref{eq:N_d_est} implies
\begin{align*}
    \Prob{\inf_{|v|=1}\langle v , \mathcal{A}v\rangle \leq \zeta; |\mathcal{A}|&\leq \tfrac{1}{\zeta}} \leq \mathbb{P}(\cup_{k=1}^{N_d} \{\langle u_k , \mathcal{A}u_k\rangle < 2\zeta \} )\\
    &\leq \sum_{k=1}^{N_d} (2\zeta)^{p+2d} \E{|\langle u_k , \mathcal{A}u_k\rangle|^{-(p+2d)}}\\
     &\leq N_d (2\zeta)^{p+2d}\sup_{|v|=1} \E{|\langle v , \mathcal{A}v\rangle|^{-(p+2d)}} \\
     & \leq 2^{p+2d}C_d \zeta^{p}\sup_{|v|=1} \E{|\langle v , \mathcal{A}v\rangle|^{-(p+2d)}}.
\end{align*}
The above display, (\ref{eq:smallest_eigenvalue_est_1}) and (\ref{eq:smallest_eigenvalue_est_2}) show that there is $C_{p,d}>0$ depending only on~$p$ and~$d$ such that~\eqref{eq:smallest_eig} holds.
\end{proof}

\bibliographystyle{alpha}
\end{document}